\documentclass[12pt,a4paper]{article}
\usepackage[utf8]{inputenc}
\usepackage{amsmath}
\usepackage{amsfonts}
\usepackage{amssymb}
\usepackage{graphicx}
\usepackage{mathrsfs}
\usepackage{eqnarray,amsthm}
\usepackage{lmodern}
\usepackage[T1]{fontenc}
\usepackage{fullpage}
\usepackage{enumitem}

\begin{document}

\title{\textbf{Classification of nonlinear boundary conditions for 1D nonconvex Hamilton-Jacobi equations}}

\author{Jessica Guerand\footnote{ D\'epartement de Mathématiques et applications, \'Ecole Normale Sup\'erieure (Paris), 
45 rue d’Ulm, 75005 Paris, France. \texttt{jessica.guerand@ens.fr}}}

\maketitle

\newtheorem{thm}{\bf Theorem}[section]
\newtheorem{prop}[thm]{Proposition}
\newtheorem{defin}[thm]{Definition}
\newtheorem{Lm}[thm]{Lemma}
\newtheorem{cor}[thm]{Corollary}
\newtheorem{theo}[thm]{Theorem}
\theoremstyle{remark}
\newtheorem{Rmk}[thm]{Remark}

\begin{center}
\textbf{Abstract}
\end{center}

We study Hamilton-Jacobi equations in $[0,+\infty)$ of evolution type with nonlinear boundary conditions of Neumann type in the case where the Hamiltonian is non necessarily convex with respect to the gradient variable. In this paper, we give two main results. First, we prove a classification of boundary condition result for a nonconvex, coercive Hamiltonian, in the spirit of the flux-limited formulation for quasi-convex Hamilton-Jacobi equations on networks recently introduced by Imbert and Monneau. 
 Second, we give a comparison principle for a nonconvex and noncoercive Hamiltonian where the boundary condition can have flat parts.
\bigbreak
\bigbreak

\section{Introduction}

\subsection{Hamilton-Jacobi equation and flux-limited solutions}

This paper deals with Hamilton-Jacobi equations of the type $$\left\{ \begin{array}{lllll}
u_{t}+H(u_{x})= 0  & \mbox{ for } & t \in (0,T) & \mbox{ and }& x>0\\
u_{t}+F(u_{x})= 0 & \mbox{ for } & t \in (0,T) & \mbox{ and }& x=0,
\end{array}\right.$$ for $T>0$, associated with a nonconvex and noncoercive (only for one result) Hamiltonian in the gradient variable. Imbert and Monneau prove in \cite{im,im2}, two mains results, among others. First, they prove a comparison principle for quasi-convex Hamilton-Jacobi equations on networks. Second, they give a classification result, imposing a general junction condition reduce to imposing a junction condition of optimal control type (see also \cite{gue1}), here a flux-limited junction condition.  
 The purpose of this paper is to obtain the results of Imbert and Monneau for a nonconvex Hamiltonian on the half line $[0,+\infty)$. 

\vspace{0.5cm}

\textbf{Comparison with known results.} First we deal with known results about comparison principles. 
There exist many results for Hamilton-Jaco\-bi equations with boundary conditions of Neumann type. 
In \cite{PLL}, the author studied the case of linear Neumann boundary condition.
For first-order Hamilton-Jacobi equations, Barles and Lions prove a comparison principle result in \cite{barleslionsordre1} under a nondegeneracy condition on the boundary nonlinearity (see \eqref{decres} below). The second-order case was treated by Ishii and Barles in \cite{ishordre2,Barlesellip,barles1}.
More precisely, Barles proves in \cite{barles1} a comparison principle for fully non linear second order, degenerate, parabolic equations, in a smooth subset $\Omega$ of $\mathbb{R}^{N}$, i.e.,
$$u_{t}+H(x,u,Du,D^{2}u)=0 \mbox{ in } \Omega,$$
with a nonlinear Neumann boundary condition satisfying the same nondegeneracy as in \cite{barleslionsordre1},
$$u_{t}+F(x,u,Du)=0 \mbox{ in } \Omega.$$
In this paper, we restrict ourselves to the case where $H$ and $F$ only depends on the gradient variable.
In \cite{barles1,barleslionsordre1}, considering only the gradient variable dependence, the boundary condition satisfies 
\begin{equation}
\label{decres}
F(p-\lambda)-F(p)\geq C\lambda, \quad  \mbox{ for } \lambda>0.
\end{equation}
 In this paper we assume a more general boundary condition, here $F$ is non-increasing, possibly with flat parts, and satisfies  
$$\lim_{p\rightarrow -\infty}F(p)=+\infty \quad \mbox{ and } \lim_{p\rightarrow +\infty}F(p)=-\infty.$$
For example, the function $F(p)=-\mbox{argsh} (p)$ does not satisfy the first condition but satisfies the second one. 

In \cite{lionssoug}, the authors deal with nonconvex coercive Hamiltonians on junctions. They prove a comparison principle for this state constraint problem (here, we write it in the case where the Hamiltonians only depend on the gradient variable and the junction is reduced to one branch i.e., a half-line), 
\begin{equation}
\label{LionsSoug}
\begin{array}{lll}
u_{t}+H(u_{x}) & = 0 & \mbox{in } (0,T)\times (0,+\infty)\\
u_{t}+H(u_{x}) &\geq 0 & \mbox{in } (0,T)\times \{ 0 \}.
\end{array}
\end{equation}
This problem is an extension to the state constraint problem of Soner \cite{son} and Ishii and Koike \cite{ish}, where the authors study the case of a convex Hamiltonian. 
For $H$ quasi-convex, in \cite{im}, the authors prove that \eqref{LionsSoug} is equivalent to 
\begin{equation}
\label{IMSC}
\begin{array}{lll}
u_{t}+H(u_{x}) & = 0 & \mbox{in } (0,T)\times (0,+\infty)\\
u_{t}+H^{-}(u_{x}) &= 0 & \mbox{in } (0,T)\times \{ 0 \},
\end{array}
\end{equation}
where $H^{-}$ is the decreasing part of the Hamiltonian, see also \cite{gue1} for the multidimensional case. 
If we define for $H$ nonconvex, 
$$H^{-}(p)=\inf\limits_{q\leq p} H(q),$$
one can prove the equivalence between \eqref{LionsSoug} and \eqref{IMSC} using the same methods as in \cite{im,gue1} and results of this paper (see Appendix A). For a junction with many branches, one can get the same kind of equivalence of equations with the same tools.
 In this paper, we get a comparison principle for \eqref{IMSC} and more generally, not only for $H^{-}$, but for any continuous, non-increasing, semi-coercive function. 
 
As far as classification of boundary conditions are concerned, in a pioneer work Andreianov and Sbihi \cite{andreiaSbi1,andreiaSbi2,andreiaSbi3} are able to describe effective boundary conditions for scalar conservation laws. Concerning the Hamilton-Jacobi framework, first results were obtained for quasi-convex Hamiltonians by Imbert and Monneau. They treat the problem on a junction with several branches in 1D \cite{im} and in the multi-dimensional case \cite{im2}.
Still in a quasi-convex framework, the authors in \cite{vinhimbert} prove a classification result of more general boundary conditions for degenerate parabolic equations. The nonconvex case has been out of reach so far.
In this paper, we get a classification result for a nonconvex Hamiltonian in 1D on the half-line. Monneau proves independently in \cite{regis} a classification result for a nonconvex Hamiltonian in the multi-dimensional case on a junction. 
  
After \cite{im,im2}, many papers deal with the flux-limited formulation and results associated to the reduction of the set of test functions. These problems show the relevence of considering a more general class of boundary conditions than the classical state constraint problem \cite{son,ish} (i.e. considering $F_{A}$ that is more general than $H^{-}$). Homogenisation results have been recently obtained in \cite{GIM,FS}.  Moreover, there have been numerical results for a quasi-convex Hamiltonian and a flux-limited function at the junction point. There is a convergence result for a flux-limited function at the junction point in \cite{CLM}.  In \cite{IK}, the authors find an error estimate of order $\Delta x^{\frac{1}{3}}$ of the same scheme as in \cite{CLM}, and prove a convergence result for a general junction function at the junction point. This error estimate has been improved in \cite{GK} to order $\Delta x^{\frac{1}{2}}$. There are also applications in optimal control, for example in \cite{AOT} where the authors study problem related to flux-limited functions. 
  
\vspace{0.5cm}

\textbf{Contributions of the paper.} In this article, as in \cite{im} for quasi-convex Hamiltonians, we prove first that boundary conditions can be also classified for a nonconvex coercive Hamiltonian by generalizing the definition of $A$-limited flux. 
Second, we prove first a comparison principle for a nonconvex and noncoercive Hamiltonian where the boundary condition can have flat parts. The main idea  of the proof is to replace the classical term of the doubling variable method $\frac{(t-s)^{2}}{2\delta}+\frac{(x-y)^{2}}{2\epsilon}$ by an appropriate function coupling time and space $\delta\varphi\left(\frac{t-s}{\delta},\frac{x-y}{\delta}\right)$ which prevents the classical supremum to be reached at the boundary.

\vspace{0.5cm}

\textbf{Comments and difficulties.}
For the classification result, the main difficulty was to find the good definition of flux-limited function $F_{A}$ for a nonconvex coercive Hamiltonian. In \cite{im}, for a quasi-convex Hamiltonian, Imbert and Monneau prove that boundary conditions can be classified with the flux-limited functions of the following form (see figure \ref{convexFA})
$$F_{A}(p)=\max(A,H^{-}(p)),$$ 
which are also BLN flux functions (see \cite{BLN}) defined as, for $p_{0}\in \mathbb{R}$, 
$$F_{p_{0}}(p)= \left\{\begin{array}{ll}
\sup\limits_{q\in [p,p_{0}]} H(q) & \mbox{if } p\leq p_{0}\\
\inf\limits_{q\in [p_{0},p]} H(q) & \mbox{if } p\geq p_{0}.
\end{array}\right.$$ The BLN flux functions can be defined for nonconvex Hamiltonians.
However, in the nonconvex case, BLN flux functions are not sufficient to classify boundary conditions. For example, for an Hamiltonian with two minima (see figure \ref{nonconvFA}), we need flux-limited functions with two flat parts $A_{1}$ and $A_{2}$ like in figure \ref{nonconvFA}, but this function is not a BLN flux function. However, it is locally a BLN function. In fact it is the ``effective'' boundary condition introduced in \cite{andreiaSbi1,andreiaSbi2,andreiaSbi3}. As we only have a comparison result for the half line case, we only give the proof of the classification result in the half line case. However, a different approach dealing with $N$ branches in the multi-dimensional case is developped in \cite{regis}.

 For the comparison principle, we tried to generalize the idea of Imbert and Monneau in \cite{im} of the ``vertex test function''. In their comparison principle, they replaced the classical term $\frac{(x-y)^{2}}{2\varepsilon}$ by a function $G$ called the ``vertex test function'' which satisfies (almost) the following condition 
$$H(y,-G_{y})\leq H(x,G_{x}),$$
which gives a contradiction combining the two viscosity inequalities.  But for nonconvex Hamiltonians even for a junction with only one branch, it is very difficult to find such a ``vertex test function''. However, we follow the idea of coupling time and space in the doubling variable method in \cite{FIM}. For example for the boundary condition $F(p)=H(0,p)=-p$, taking $$\frac{(t-s)^{2}}{2\delta}+\frac{(t-s)}{\delta}(x-y)+\frac{(x-y)^{2}}{2\delta},$$
instead of the classical term $$\frac{(t-s)^{2}}{2\delta}+\frac{(x-y)^{2}}{2\delta},$$
allows to get rid of the case $x=0$ or $y=0$ in the viscosity inequalities. In this paper, we give an example of such a function coupling time and space which solves the problem for all boundary conditions satisfying, $F$ is non-increasing and
$$\lim_{p\rightarrow -\infty}F(p)=+\infty \quad \mbox{ and } \lim_{p\rightarrow +\infty}F(p)=-\infty.$$
This proof is too difficult to be adapted for a junction with several branches, that is why, this paper is written only for a half-line domain. 

\begin{figure}
\begin{center}
  \includegraphics[width=5cm]{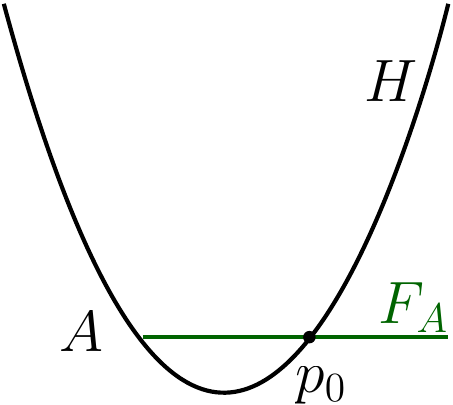}
  \caption{Illustration of the function $F_{A}$ in the convex case.}
  \label{convexFA}
  \end{center}
  \end{figure}
  
  \begin{figure}
\begin{center}
  \includegraphics[width=7.0cm]{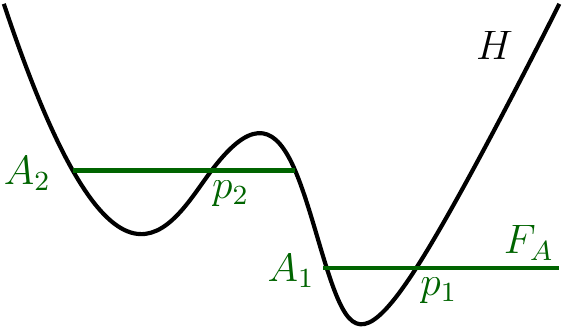}
  \caption{Illustration of a function $F_{A}$ in the nonconvex case.}
  \label{nonconvFA}
  \end{center}
  \end{figure}

\subsection{Main theorems}
 
 Let us consider the following Hamilton-Jacobi equation in $(0,T)\times [0,+\infty)$ 
 \begin{equation}
\label{eqHJ}
\left\{ \begin{array}{lllll}
u_{t}+H(u_{x})= 0  & \mbox{ for } & t \in (0,T) & \mbox{ and }& x>0\\
u_{t}+F(u_{x})= 0 & \mbox{ for } & t \in (0,T) & \mbox{ and }& x=0
\end{array}\right.
\end{equation}
subject to the initial condition

\begin{equation}
\label{IC}
u(0,x)=u_{0}(x) \quad \mbox{ for } \quad x\geq 0.
\end{equation}
We study the case of a continuous Hamiltonian $H:\mathbb{R}\rightarrow \mathbb{R}$ and a continuous non-increasing function $F:\mathbb{R}\rightarrow \mathbb{R},$ which satisfy other properties specified in the theorems.
In this paper, we don't prove any existence result, as the proof of  \cite[Theorem 2.14]{im} prove also the existence of a solution in our case, for a nonconvex and noncoercive Hamiltonian. 
Let us state our main theorem, the classification result, which is the extension of  \cite[Theorem 1.1]{im} to the case of a nonconvex Hamiltonian. 
  
  \begin{figure}
\begin{center}
  \includegraphics[width=10.0cm]{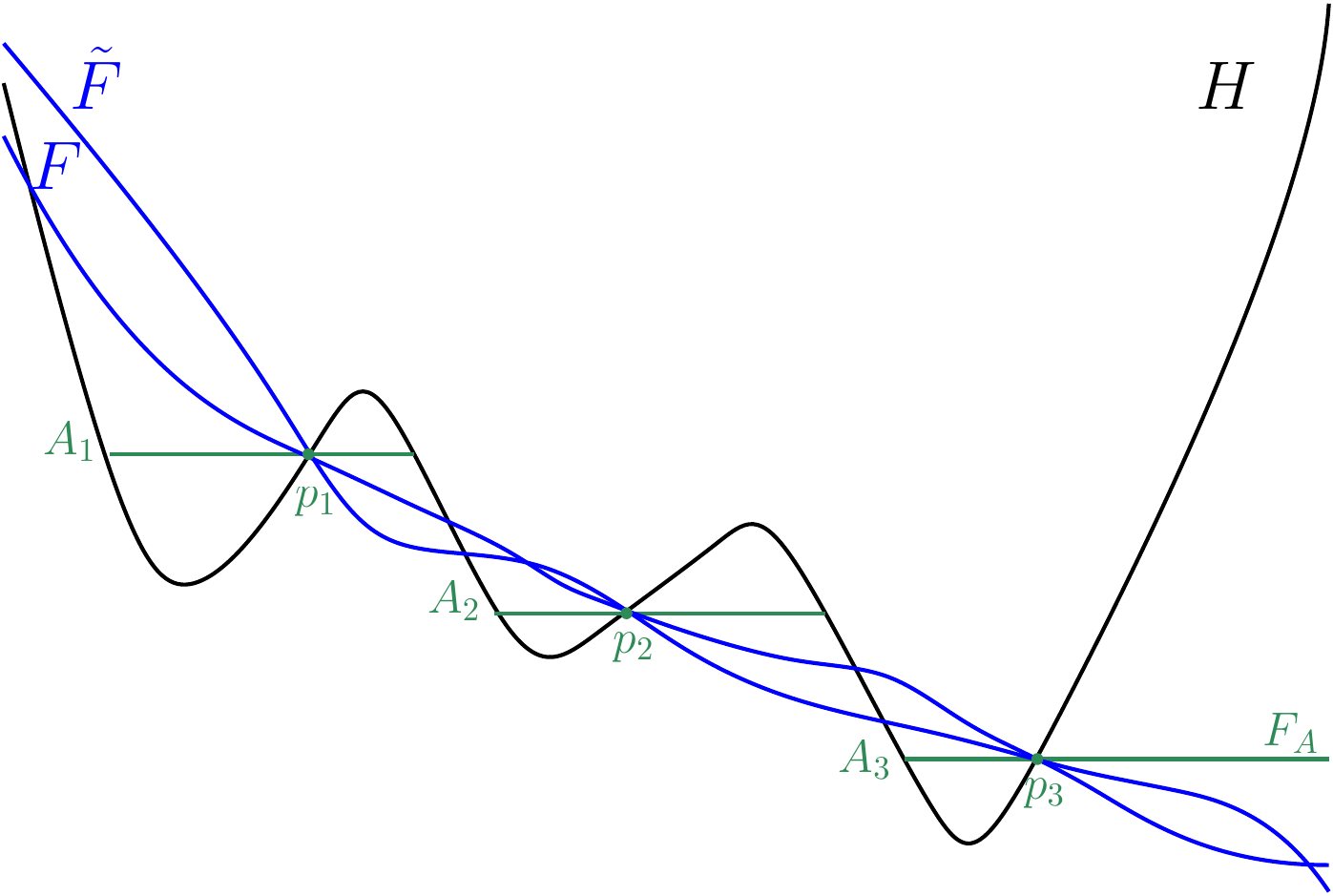}
  \caption{Illustration of a function $F_{A}$ associated to $F$ in Theorem \ref{thclassification}.}
  \label{thclass}
  \end{center}
  \end{figure}

  To understand the result, we comment it on an example, see Figure \ref{thclass}. The following theorem gives the equivalence between the relaxed equation of \eqref{eqHJ} for a general $F$ and the equation \eqref{eqHJ} for $F=F_{A}$, where $F_{A}$ is a non-increasing function which is ``almost'' the function $H$ where each non-decreasing part are replaced by the ``right constant''. In the particular case of Figure \ref{thclass}, the ``right constants'' are given by the intersection of $F$ and the non-decreasing parts of $H$. We deduce here that taking $\tilde{F}$ instead of $F$ gives the same solutions of the relaxed equation of \eqref{eqHJ}. The flux function $F_{A}$ and the set limiter $A_{F}$ are defined in part 3 of this paper. The definition of relaxed solutions and flux-limited solutions are given in part 2. 

\begin{theo}[Classification of general Neumann boundary conditions]
\label{thclassification}
Assume that the Hamiltonian $H:\mathbb{R}\rightarrow\mathbb{R}$ is continuous and coercive
\begin{equation}
\label{coerc}
 \lim\limits_{|p|\rightarrow+\infty}H(p)=+\infty,
 \end{equation}
 the function $F:\mathbb{R}\rightarrow\mathbb{R}$ is continuous, non-increasing and semi-coercive 
 \begin{equation}
\label{propF2}
\lim_{p\rightarrow -\infty}F(p)=+\infty.
\end{equation}
 Then there exists a unique set limiter $A_{F}$ (cf. Definition \ref{defFAF}) such that any relaxed solution of (\ref{eqHJ}) is in fact a flux-limited solution of (\ref{eqHJ}) with $F=F_{A_{F}}$.
\end{theo}

\begin{Rmk}
R. Monneau developed independently in \cite{regis} a different approach, in particular, he can deal with the multi-dimensional case for a junction with several branches.
\end{Rmk}

Now let us state the comparison principles.

\begin{theo}[Comparison principles] 
\label{mainthCPnc}
Assume that the Hamiltonian $H:\mathbb{R}\rightarrow\mathbb{R}$ is continuous, the function $F:\mathbb{R}\rightarrow\mathbb{R}$ is continuous, non-increasing and semi-coercive \eqref{propF2}
and the initial datum $u_{0}$ is uniformly continuous.
Moreover, if we have one of the following assumptions, 
\begin{enumerate}
\item (a noncoercive Hamiltonian and a ``coercive'' flux function)
\begin{equation}
\label{propF}
\lim_{p\rightarrow +\infty}F(p)=-\infty,
\end{equation}
\item (a coercive Hamiltonian and a semi-coercive flux function)

$$ \lim\limits_{|p|\rightarrow+\infty}H(p)=+\infty.$$
\end{enumerate}

Then for all (relaxed) sub-solution $u$ and (relaxed) super-solution $v$ of (\ref{eqHJ})-(\ref{IC}) satisfying for some $T>0$ and $C_{T}>0,$
\begin{equation*}
u(t,x)\leq C_{T}(1+x), \quad v(t,x)\geq -C_{T}(1+x), \quad \forall (t,x)\in (0,T)\times [0,+\infty),
\end{equation*}
we have
$$u\leq v \quad \mbox{ in } \quad [0,T)\times [0,+\infty).$$
 
\end{theo}

%
%

\section{Viscosity solutions}

In this section, we recall the definitions given in \cite{im} of viscosity solutions for the relaxed and the flux-limited problem and we recall that we need a weak continuity condition for sub-solutions.
\subsection{Relaxed and flux-limited solutions}
Here the class of test functions on $(0,T)\times [0,+\infty)$ is $\mathcal{C}^{1}$.
We say that a test function $\phi$ touches a function $u$ from below (resp. from above) at $(t,x)$ if $u-\phi$ reaches a local minimum (resp. maximum) at $(t,x)$. 

We recall the definition of upper and lower semi-continuous envelopes $u^{*}$ and $u_{*}$ of a (locally bounded) function $u$ defined on $[0,T)\times [0,+\infty)$,
$$ u^{*}(t,x)=\limsup_{(s,y)\rightarrow (t,x)} u(s,y) \quad \mbox{and} \quad u_{*}(t,x)=\liminf_{(s,y)\rightarrow (t,x)} u(s,y).$$ 

\begin{defin}[Relaxed solutions]
Let $u:[0,T)\times [0,+\infty)\rightarrow\mathbb{R}$.
\begin{enumerate}[label=\roman*)]
\item We say that $u$ is a \emph{relaxed sub-solution} (resp. \emph{relaxed super-solution}) of \eqref{eqHJ} in $(0,T)\times [0,+\infty)$ if for all test function $\phi\in \mathcal{C}^{1}$ touching $u^{*}$ (resp. $u_{*}$) from above (resp. from below) at $(t_{0},x_{0})$, we have if $x_{0}>0$,
$$\phi_{t}(t_{0},x_{0})+H(\phi_{x}(t_{0},x_{0}))\leq 0 \quad \mbox{(resp. } \geq 0 \mbox{)}$$
if $x_{0}=0$, 
$$\begin{array}{lll}
\mbox{either} & \phi_{t}(t_{0},0)+H(\phi_{x}(t_{0},0))\leq 0 & \mbox{(resp. } \geq 0 \mbox{)}\\
\mbox{or} & \phi_{t}(t_{0},0)+F(\phi_{x}(t_{0},0))\leq 0 & \mbox{(resp. } \geq 0 \mbox{)}.
\end{array} 
$$
\item We say that $u$ is a \emph{relaxed sub-solution} (resp. \emph{relaxed super-solution}) of \eqref{eqHJ}-\eqref{IC} on $[0,T)\times [0,+\infty)$ if additionally
$$u^{*}(0,x) \leq u_{0}(x) \quad \mbox{(resp. } \quad u_{*}(0,x) \geq u_{0}(x)  \mbox{)} \quad \forall x\in [0,+\infty).$$
\item We say that $u$ is a \emph{relaxed solution} if $u$ is both a relaxed sub-solution and a relaxed super-solution.
\end{enumerate} 
\end{defin}

Let us recall the definition of flux-limited solutions given in \cite{im}.
\begin{defin}[Flux-limited solutions]
Let $u:[0,T)\times [0,+\infty)\rightarrow\mathbb{R}$.
\begin{enumerate}[label=\roman*)]
\item We say that $u$ is a \emph{flux-limited sub-solution} (resp. \emph{flux-limited super-solution}) of \eqref{eqHJ} in $(0,T)\times [0,+\infty)$ if for all test function $\phi\in \mathcal{C}^{1}$ touching $u^{*}$ (resp. $u_{*}$) from above (resp. from below) at $(t_{0},x_{0})$, we have if $x_{0}>0$,
$$\phi_{t}(t_{0},x_{0})+H(\phi_{x}(t_{0},x_{0}))\leq 0 \quad \mbox{(resp. } \geq 0 \mbox{)}$$
if $x_{0}=0$, 
$$\phi_{t}(t_{0},0)+F(\phi_{x}(t_{0},0))\leq 0 \quad \mbox{(resp. } \geq 0 \mbox{)}.$$
\item We say that $u$ is a \emph{flux-limited sub-solution} (resp. \emph{flux-limited super-solution}) of \eqref{eqHJ}-\eqref{IC} on $[0,T)\times [0,+\infty)$ if additionally
$$u^{*}(0,x) \leq u_{0}(x) \quad \mbox{(resp. } \quad u_{*}(0,x) \geq u_{0}(x)  \mbox{)} \quad \forall x\in [0,+\infty).$$
\item We say that $u$ is a \emph{flux-limited solution} if $u$ is both a flux-limited sub-solution and a flux-limited super-solution.
\end{enumerate} 
\end{defin}

\subsection{``Weak continuity''  condition for sub-solutions}
For the same reason as in \cite{im}, we need a weak continuity condition for sub-solutions to get the classification result in section 4. Let us recall that any relaxed sub-solution satisfies automatically the ``weak continuity'' condition if the function $F$ is semi-coercive, that is to say if $F$ satisfies \eqref{propF2}. Precisely, we recall the \cite[Lemma 2.3]{im} without proving it since the proof is the same in our case. 

\begin{Lm} [``Weak continuity'' condition] 
Assume that the Hamiltonian $H:\mathbb{R}\rightarrow\mathbb{R}$ is continuous and coercive, the function $F:\mathbb{R}\rightarrow\mathbb{R}$ is continuous, non-increasing and semi-coercive.
Then any relaxed sub-solution $u$ of \eqref{eqHJ} satisfies for all $t\in (0,T)$ 
$$u(t,0)=\limsup\limits_{(s,y)\rightarrow (t,0), y>0} u(s,y).$$
\end{Lm}

\section{Classification of boundary conditions}
In this section, we extend the definitions from \cite{im} of the flux limiter $A$ and the $A$-limited flux function $F_{A}$ to nonconvex coercive Hamiltonians. We obtain the same result of reduction of the set of test functions for the $A$-limited flux functions and the classification result. We show that only the Hamiltonian $H$ and few points of the function $F$ characterize the boundary conditions. Using the result of the fourth section, we prove that the solution of the problem (\ref{eqHJ})-(\ref{IC}) is unique. 

In this section, the Hamiltonian $H:\mathbb{R}\rightarrow \mathbb{R}$ is assumed to be continuous and coercive \eqref{coerc}.

\subsection{Set limiters and limited flux functions}
As for quasi-convex Hamiltonians in \cite{im}, we construct a flux function $F_{A}$ which is constant on some subsets of $\mathbb{R}$.
First, let us give some definitions and lemmas which are used to define the function $F_{A}$.

\begin{figure}
\begin{center}
  \includegraphics[width=7.0cm]{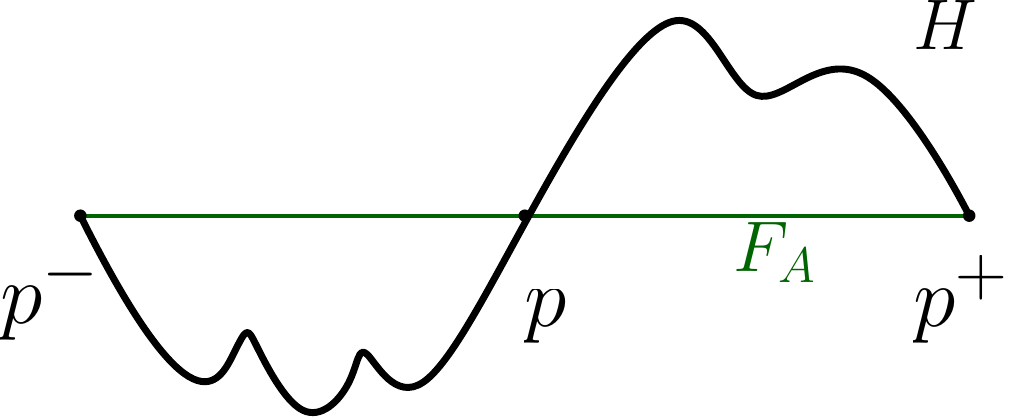}
  \caption{Illustration of $p^{-}$ and $p^{+}$ in Definition \ref{defpm1}}
  \label{defpm}
  \end{center}
  \end{figure}

\subsubsection{Numbers $p^{-}$ and $p^{+}$}

\begin{defin}[Numbers $p^{-}$ and $p^{+}$]
\label{defpm1}
Let $p\in\mathbb{R}.$ We define
$$p^{-}=\sup\left\{  q<p \mbox{  } | \quad H(q)\geq H(p) \right\},$$ and
$$p^{+}=\inf\left\{  q>p \mbox{  } | \quad H(q)\leq H(p) \right\},$$
with the convention $\inf \emptyset =+\infty$.
\end{defin}

\begin{Rmk}
As the Hamiltonian $H$ is coercive, $p^{-}$ is the supremum of a nonempty set.
\end{Rmk}

We deduce the following lemma from the definition.

 \begin{Lm} 
 \label{H-+}
 For all $p\in \mathbb{R}$, we have $$H(p^{-})=H(p)=H(p^{+}).$$
 Moreover, we have
 \begin{equation}
 \label{souspmoins}
 \forall q\in]p^{-},p[, \quad H(q)<H(p),
 \end{equation}
 and
  \begin{equation}
 \label{surpplus}
 \forall q\in]p,p^{+}[, \quad H(q)>H(p).
 \end{equation}
 \end{Lm} 
 
 \begin{proof}[Proof of Lemma \ref{H-+}]
 The second part of the lemma is a consequence of the definition of $p^{-}$ and $p^{+}$.
 Let us prove the first part.
 By definition, we have $H(p^{-})\geq H(p)$ and $\forall q \in ]p^{-},p[, \mbox{ } H(q)<H(p)$. Sending $q\rightarrow p^{-}$ and by continuity of $H$, we deduce $H(p^{-})\leq H(p)$ so $H(p^{-})=H(p)$. By the same arguments, we have  $H(p)=H(p^{+})$.
 \end{proof}

On Figure \ref{defpm}, the position of $H$ compared to $H(p)$ is illustrated.

Let us give the following useful lemma.

\begin{Lm}
 \label{pmoins} We have the following properties.
\begin{enumerate}
\item Assume $]p^{-},p[\cap ]q^{-},q[\neq \emptyset$. We have $H(p)\leq H(q)$ if and only if $[p^{-},p]\subset [q^{-},q]$ i.e., $q^{-}\leq p^{-}<p\leq q$.
\item Assume $]p,p^{+}[\cap ]q,q^{+}[\neq \emptyset$. We have $H(p)\leq H(q)$ if and only if $[q,q^{+}]\subset [p,p^{+}]$ i.e., $p\leq q<q^{+}\leq p^{+}$.
\item If $]p^{-},p[\cap ]q,q^{+}[\neq\emptyset$, then $H(p)>H(q)$.
\end{enumerate}
 \end{Lm}
 
 \begin{proof}[Proof of Lemma \ref{pmoins}]
 Let us prove the first point. The second point is very similar to the first one so we skip the proof.
 Assume that $H(p)\leq H(q)$. If by contradiction $p>q$, then since $]p^{-},p[\cap ]q^{-},q[\neq \emptyset$, we have $p^{-}<q<p$. We deduce that $$H(q)<H(p)\leq H(q)$$ which gives a contradiction. So we deduce that $p\leq q$. Moreover, since $]p^{-},p[\cap ]q^{-},q[\neq \emptyset$, we have $q^{-}<p\leq q$. Assume by contradiction that $p^{-}<q^{-}$, then $$H(p^{-})=H(p)\leq H(q)=H(q^{-}),$$
 but $q^{-}\in ]p^{-},p[$, which gives a contradiction with Lemma \ref{H-+}.
 So we deduce that $[p^{-},p]\subset [q^{-},q]$.
 Assume now that $[p^{-},p]\subset [q^{-},q]$. In particular we have $p\in [q^{-},q]$, hence $H(p)\leq H(q)$.
   
Let us prove the third point. 
Assume that 
\begin{equation}
\label{pq}
]p^{-},p[\cap ]q,q^{+}[\neq\emptyset,
\end{equation}
then we have $q\leq p$. 
 Necessarily by Lemma \ref{H-+}, we have $H(p)\geq H(q)$. 
 If by contradiction, we have $H(p)=H(q)$, then either $q=p$ so $q^{-}=p^{-}$ or $q\leq p^{-}$ so $q^{+}\leq p^{-}$. But these two cases gives a contradiction with \eqref{pq}. So we deduce that $H(p)>H(q)$.
 \end{proof}

\subsubsection{Set limiters and limited flux functions}

\begin{defin}[Set limiter $A$]
\label{defA}
The set $A$ is called a \emph{set limiter} if $A$ is a set of points of $\mathbb{R}$ indexed by $I$, $A=(p_{\alpha})_{\alpha\in I}$, such that 
\begin{enumerate}
\item $\forall \alpha \in I$, $p_{\alpha}^{-}\neq p_{\alpha}^{+},$
\item For $\alpha_{1}, \alpha_{2} \in I$, if $p_{\alpha_{1}}<p_{\alpha_{2}}$ then $H(p_{\alpha_{1}})\geq H(p_{\alpha_{2}}),$
\item \begin{itemize}
\item $\forall p\in \mathbb{R}$ such that $p^{-}<p$, $\exists\alpha\in I$ such that $]p^{-},p[\cap ]p_{\alpha}^{-},p_{\alpha}^{+}[\neq \emptyset$,
\item $\forall p\in \mathbb{R}$ such that $p<p^{+}$, $\exists\alpha\in I$ such that $]p,p^{+}[\cap ]p_{\alpha}^{-},p_{\alpha}^{+}[\neq \emptyset$.
\end{itemize}
\end{enumerate}
\end{defin}

\begin{Rmk} $A$ is not empty as the Hamiltonian $H$ is coercive.
\end{Rmk}

We deduce the following lemma which allows to define the flux function.
 \begin{Lm}
 \label{disjoint}
 If $p_{1}<p_{2}$ and $H(p_{1})\geq H(p_{2})$ then we have $]p_{1}^{-},p_{1}^{+}[\cap]p_{2}^{-},p_{2}^{+}[=\emptyset$.
 In particular, the intervals $]p_{\alpha}^{-},p_{\alpha}^{+}[$ for $\alpha\in I$ are disjoint.
 \end{Lm}
 
 \begin{proof}[Proof of Lemma \ref{disjoint}]
This lemma is a direct consequence of Lemma \ref{pmoins}.
 \end{proof}

Now we can define the $A$-limited flux function.

\begin{defin}[Function $F_{A}$]
\label{funcFA}
Let $A$ be a set limiter. The function $F_{A}:\mathbb{R}\rightarrow\mathbb{R}$ defined by 
$$F_{A}(p)=\left\{ \begin{array}{ll}
H(p_{\alpha}) & \mbox{ if } p\in [p_{\alpha}^{-},p_{\alpha}^{+}], \mbox{ for } \alpha\in I\\
H(p) & \mbox{ elsewhere}
\end{array}\right.$$
is called a \emph{$A$-limited flux function}.
\end{defin} 

\begin{prop}
\label{FAdec}
The function $F_{A}$ is well-defined, continuous and non-increasing.
\end{prop}

\begin{figure}
\begin{center}
  \includegraphics[width=15.5cm]{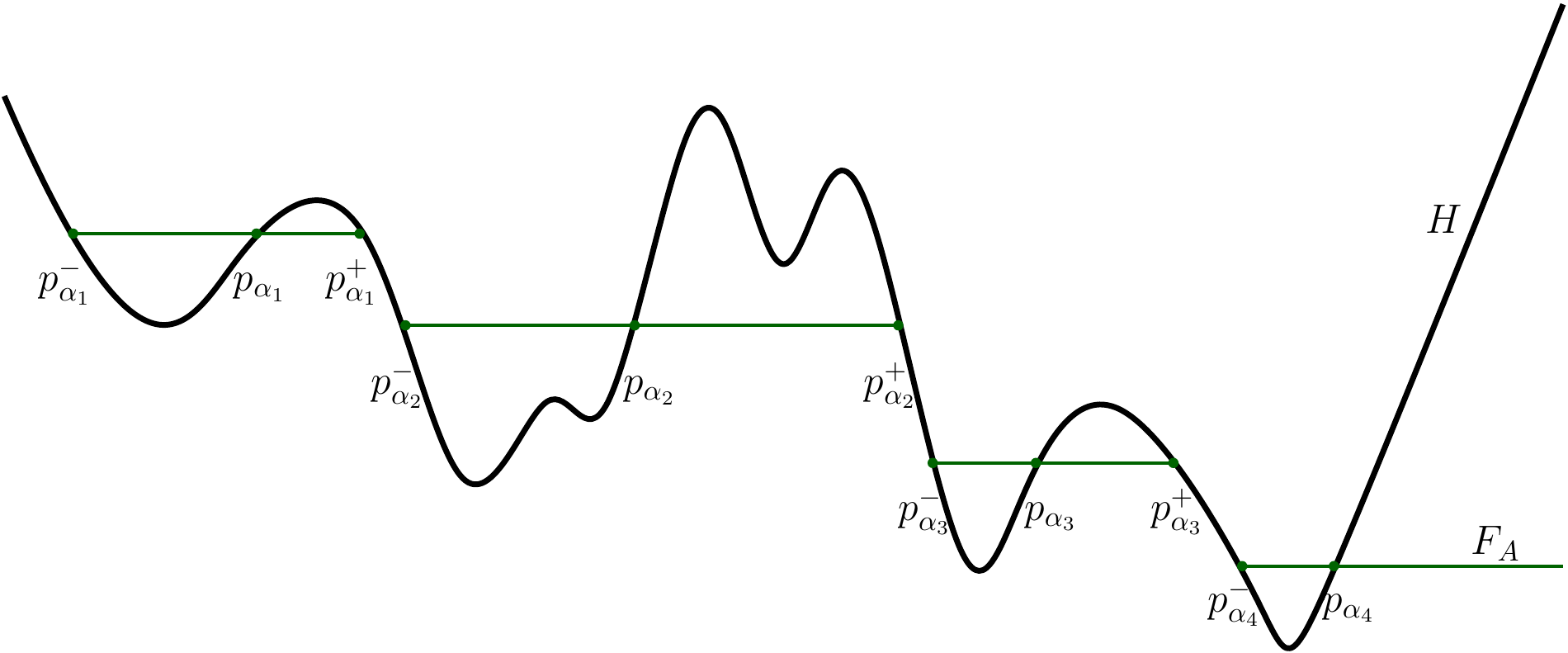}
  \caption{Illustration of a function $F_{A}$ in Definition \ref{funcFA}}
  \label{figFA}
  \end{center}
  \end{figure}

We give an example of a $A$-limited flux in Figure \ref{figFA}.

 \begin{proof}[Proof of Proposition \ref{FAdec}] 
Lemma \ref{disjoint} ensures that the function $F_{A}$ is well-defined and Lemma \ref{H-+} ensures that $F_{A}$ is continuous.
 Let us prove that $F_{A}$ is non-increasing. 
 Assume by contradiction that there exists $p<q$ such that $F_{A}(p)<F_{A}(q)$.  Without loss of generality, we assume that $p<q$ such that $H(p)=F_{A}(p)<F_{A}(q)=H(q)$. Indeed, if we have $p\in [p_{\alpha}^{-}, p_{\alpha}^{+}]$ for $\alpha\in I$, we also have $p_{\alpha}<q$ and  $H(p_{\alpha})=F_{A}(p_{\alpha})=F_{A}(p)<F_{A}(q)$. We can use the same argument for $q$, if $q\in [p_{\alpha'}^{-}, p_{\alpha'}^{+}]$ for $\alpha'\in I$. 
 
 Let $p_{1}=\inf\left\{ r\geq p \mbox{ } | \mbox{ } H(r)=\frac{H(p)+H(q)}{2} \right\}$ and $q_{1}=\sup\left\{ r\leq q \mbox{ } | \mbox{ } H(r)=\frac{H(p)+H(q)}{2} \right\}.$ We have $$p_{1}^{-}<p<p_{1}\leq q_{1} <q <q_{1}^{+},$$ and 
\begin{equation}
\label{p1q1}
H(p)<H(p_{1})=H(q_{1})<H(q).
\end{equation} 
 Using $3.$ of Definition \ref{defA}, there exists $\alpha\in I$ such that
 $$]p_{1}^{-}, p_{1}[\cap]p_{\alpha}^{-},p_{\alpha}^{+}[ \neq \emptyset.$$
 We distinguish two cases. 
 
 If  $]p_{1}^{-}, p_{1}[\cap]p_{\alpha}^{-},p_{\alpha}[ \neq \emptyset,$ then using $1.$ of Lemma \ref{pmoins}, we deduce $H(p_{\alpha})<H(p_{1})$ and $p_{\alpha}<p_{1}$. Indeed, if by contradiction we have $H(p_{\alpha})
 \geq H(p_{1})$, then by $1.$ of Lemma \ref{pmoins}, we deduce that $p\in [p_{1}^{-},p_{1}]\subset [p_{\alpha}^{-},p_{\alpha}]$.
 Hence, we have $$H(p)=F_{A}(p)=F_{A}(p_{\alpha})=H(p_{\alpha})\geq H(p_{1}),$$ which gives a contradiction with \eqref{p1q1}.
 We deduce that $$H(p_{\alpha})=F_{A}(p_{\alpha})<H(p_{1})$$ and $[p_{\alpha}^{-},p_{\alpha}]\subset [p_{1}^{-},p_{1}]$ with $1.$ of Lemma \ref{pmoins}, hence $p_{\alpha}<p_{1}.$
 
 If $]p_{1}^{-}, p_{1}[\cap]p_{\alpha},p_{\alpha}^{+}[ \neq \emptyset,$ then $p_{\alpha}<p_{1}$ and using $3.$ of Lemma \ref{pmoins}, we deduce that 
 $$H(p_{\alpha})=F_{A}(p_{\alpha})<H(p_{1}).$$
 
 By symmetric arguments, we also have $\alpha'\in I$ such that 
 $$H(p_{\alpha'})=F_{A}(p_{\alpha'})>H(q_{1}),$$ and $q_{1}<p_{\alpha'}.$
 
Combining these conclusions, we deduce that $$p_{\alpha}<p_{1}<q_{1}<p_{\alpha'},$$ and $$H(p_{\alpha})<H(p_{1})=H(q_{1})<H(p_{\alpha'}),$$ which gives a contradiction with $2.$ of Definition \ref{defA}. We deduce that $F_{A}$ is non-increasing.
 \end{proof}

We give the following lemma which is useful for the next subsection. 
 
\begin{Lm}
\label{positionHFA}
The function $F_{A}$ satisfies the following properties,
\begin{enumerate}
\item for $\alpha\in I, \quad  \forall p\in]p_{\alpha}^{-},p_{\alpha}[, \quad F_{A}(p)>H(p),$
\item for $\alpha\in I, \quad\forall p\in]p_{\alpha},p_{\alpha}^{+}[, \quad F_{A}(p)<H(p),$
\item If $ p\notin \bigcup\limits_{\alpha\in I} ]p_{\alpha}^{-},p_{\alpha}[\cup]p_{\alpha},p_{\alpha}^{+}[$, then $F_{A}(p)=H(p)$.
\end{enumerate}
\end{Lm} 

\begin{proof}
This result is a direct consequence of Lemma \ref{H-+} and Definition \ref{funcFA}. 
\end{proof}

\subsection{Reducing the set of test functions}

With this extension of definition of $F_{A}$, as in \cite{im, im2,gue1}, we can prove a theorem for reducing the set of test functions for the $A$-limited flux function. 
We consider functions satisfying a Hamilton-Jacobi equation in $(0,+\infty)$, solution of 
\begin{equation}
\label{eqR*}
u_{t}+H(u_{x})=0 \quad \mbox{ for } \quad (t,x)\in (0,T)\times (0,+\infty).
\end{equation}

\begin{theo}[Reduced set of test functions] 
\label{threduc}
Assume that the Hamiltonian $H$ is continuous and coercive \eqref{coerc}. Let $A$ be a set limiter. For all $\alpha \in A$, let us fix any time independent test function $\phi_{\alpha}(x)$ satisfying $$ \phi_{\alpha}^{\prime}(0)=p_{\alpha}.$$
Given a function $u:(0,T)\times J \rightarrow \mathbb{R}$, the following properties hold true.
\begin{enumerate}[label=\roman*)]
\item If, for $t_{0}\in (0,T)$, $u$ is an upper semi-continuous sub-solution of \eqref{eqR*} and satisfies 
\begin{equation}
\label{continuitefaible}
 u(t_{0},0)= \limsup\limits_{(s,y)\rightarrow (t_{0},x), y\neq 0} u(s,y),
\end{equation}
and if for any test function $\varphi$ touching $u$ from above at $(t_{0},0)$ with 
\begin{equation}
\label{changmfct}
\varphi (t,x) = \psi(t)+\phi_{\alpha}(x)
\end{equation}
where $\psi \in C^{1}(0,+\infty)$ and where $\alpha\in I$ is such that $p_{\alpha}^{-}\neq p_{\alpha}$, we have 
$$\varphi_{t}+F_{A}(\varphi_{x})\leq 0 \quad \mbox{ at } \quad (t_{0},0),$$
then $u$ is a $A$-flux-limited sub-solution at $(t_{0},0)$.

\item If for $t_{0}\in (0,T)$, $u$ is a lower semi-continuous super-solution of \eqref{eqR*} and 
if for any test function $\varphi$ touching $u$ from below at $(t_{0},0)$ with 
\begin{equation*}
\varphi (t,x) = \psi(t)+\phi_{\alpha}(x)
\end{equation*}
where $\psi \in C^{1}(0,+\infty)$ and where $\alpha\in I$ is such that $p_{\alpha}\neq p_{\alpha}^{+}$, we have 
$$\varphi_{t}+F_{A}(\varphi_{x})\geq 0 \quad \mbox{ at } \quad (t_{0},0),$$
then $u$ is a $A$-flux-limited super-solution at $(t_{0},0)$.
\end{enumerate}
\end{theo}

\begin{Rmk}
We only need to consider $p_{\alpha}^{-}\neq p_{\alpha}$ (resp. $p_{\alpha}\neq p_{\alpha}^{+}$) for the sub-solution (resp. super-solution) case. Indeed in $[p_{\alpha}, p_{\alpha}^{+}]$ (resp. $[p_{\alpha}^{-}, p_{\alpha}]$), the function $F_{A}$ is lower (resp. upper) than $H$ that gives directly the result, using the following Lemmas. For example, in \cite{im} for a quasi-convex Hamiltonian and for $F=F_{A_{0}}$ the decreasing part of the Hamiltonian, $A=\{ \pi^{+}(A_{0}) \}$ where $H(\pi^{+}(A_{0}))=A_{0}$ the minimum of $H$, we have $(\pi^{+}(A_{0}))^{-}=\pi^{+}(A_{0})$. That is why the author don't need any test function for this case in \cite[Theorem 2.7 i)]{im}. 
\end{Rmk}

To prove this result, we need the two following lemmas already proven in \cite{im, im2, gue1}. Here we skip the proof on these lemmas. 

\begin{Lm}[Critical slope for sub-solution \cite{im}]
\label{Tpentescri}
Let $u$ be an upper semi-continuous sub-solution of \eqref{eqR*} which satisfies \eqref{continuitefaible}
 and let $\varphi$ be a test function touching $u$ from above at some point $(t_{0},0)$ where $t_{0}\in (0,T)$. Then the critical slope given by 
$$\bar{p}= \inf\left\{
p\in \mathbb{R}: \exists r >0, \quad \varphi(t,x)+px \geq u(t,x), 
\quad \forall (t,x)\in (t_{0}-r,t_{0}+r)\times [0,r)
\right\}$$ 
is finite, satisfies $\bar{p} \leq 0$ and 
$$\varphi_{t}(t_{0},0)+H(\varphi_{x}(t_{0},0)+\bar{p})\leq 0.$$
\end{Lm} 

\begin{Lm}[Critical slope for super-solution \cite{im}]
\label{Tpentescri2}
Let $u$ be a lower semi-continuous super-solution of \eqref{eqR*}
 and let $\varphi$ be a test function touching $u$ from below at some point $(t_{0},0)$ where $t_{0}\in (0,T)$. If the critical slope given by 
$$\bar{p}= \sup\left\{ 
p\in \mathbb{R}: \exists r>0, \quad \varphi(t,x)+px \leq u(t,x),  
\forall (t,x)\in (t_{0}-r,t_{0}+r)\times [0,r) 
\right\}$$
is finite, then it satisfies $\bar{p} \geq 0$ and we have 
$$\varphi_{t}(t_{0},0)+H(\varphi_{x}(t_{0},0)+\bar{p})\geq 0.$$
\end{Lm} 

\begin{proof}[Proof of Proposition \ref{threduc}]
We first prove the results concerning sub-solutions.
\vskip 0.5cm
\textbf{Sub-solution.}
Let $\phi$ be a test function touching $u$ from above at $(t_{0},0)$ and let $\lambda=-\phi_{t}(t_{0},0)$. Let  $p=\phi_{x}(t_{0},0).$ We want to show that 
\begin{equation}
\label{Teqsous}
F_{A}(p)\leq \lambda.
\end{equation} Notice that by lemma \ref{Tpentescri}, there exists $\bar{p}\leq 0$ such that 
$$H(p+\bar{p})\leq \lambda.$$
As $F_{A}$ is non-increasing, we have
$$F_{A}(p)\leq F_{A}(p+\bar{p})$$ and
using Lemma \ref{positionHFA}, if $p+\bar{p}\notin \bigcup\limits_{\alpha\in I}]p_{\alpha}^{-},p_{\alpha}[$  we have 
$$F_{A}(p)\leq F_{A}(p+\bar{p})\leq H(p+\bar{p})\leq \lambda,$$ which proves the result. 

Now if $p+\bar{p}\in ]p_{\alpha}^{-},p_{\alpha}[$ for some $\alpha \in I$ such that $p_{\alpha}^{-}\neq p_{\alpha}$, then 
$$p+\bar{p} <p_{\alpha}=\phi_{\alpha}^{\prime}(0).$$
Let us consider the modified test function
$$\varphi(t,x)=\phi(t,0)+\phi_{\alpha}(x)-\phi_{\alpha}(0).$$
We have 
$$\varphi(t_{0},0)=\phi(t_{0},0)=u(t_{0},0).$$
Let us show that
\begin{equation}
\label{Texp10}
\varphi(t,x)\geq u(t,x), 
\end{equation}
on a neighborhood of $(t_{0},0).$
We have
$$p+\bar{p}=\phi_{x}(t_{0},0)+\bar{p}<\phi_{\alpha}^{\prime}(0),$$
so there exists $p_{1}$ and $p_{2}$ such that $\bar{p}<p_{1}<p_{2}$ and which satisfy
$$p+p_{i}=\phi_{x}(t_{0},0)+p_{i}<\phi_{\alpha}^{\prime}(0), \quad \forall i \in \{1,2\}.$$
As $\phi_{x}$ and $\phi_{\alpha}^{\prime}$ are continuous, on a neighborhood of $(t_{0},0),$ we have
$$\phi_{x}(t,x)+p_{i}<\phi_{\alpha}^{\prime}(x), \quad \forall i \in \{1,2\}.$$
So we have on a neighborhood of $(t_{0},0)$,
 $$\begin{array}{lll}
\phi(t,x) &=& \phi(t,0)+\displaystyle\int_{0}^{x}\phi_{x}(t,y)\mathrm{d}y \\
 & = & \varphi(t,x)+\phi_{\alpha}(0)-\phi_{\alpha}(x)+\displaystyle\int_{0}^{x}\phi_{x}(t,y)\mathrm{d}y\\
  & = & \varphi(t,x)+\displaystyle\int_{0}^{x}(\phi_{x}(t,y)-\phi_{\alpha}^{\prime}(y))\mathrm{d}y \\
  & \leq & \varphi(t,x) -p_{2}x,
\end{array}$$
and by definition of $\bar{p},$ there exists a neighborhood $(t_{0}-r,t_{0}+r)\times [0,r)$ of $(t_{0},0),$ for some $r>0$ such that 
$$\begin{array}{rll}
u(t,x) & \leq & \phi(t,x)+p_{1}x
    \\                   & \leq & \varphi(t,x)+(p_{1}-p_{2})x,\\
    & \leq & u(t,x)
\end{array}$$
so we get \eqref{Texp10}.
\newline
This test function satisfies in particular \eqref{changmfct} so we deduce that
$$ -\lambda+F_{A}(p_{\alpha})\leq 0,$$
so we have as $p+\bar{p}\in ]p_{\alpha}^{-},p_{\alpha}[$ and $F_{A}$ is constant is this interval,
$$F_{A}(p)\leq F_{A}(p+\bar{p})=F_{A}(p_{\alpha})\leq \lambda.$$
 Therefore (\ref{Teqsous}) holds true. 


Let us prove now the super-solution case.
\vskip 0.5cm
\textbf{Super-solution.} Let $\phi$ be a test function touching $u$ from below at $(t_{0},0).$ Let $\lambda=-\phi_{t}(t_{0},0),$  and  $p=\phi_{x}(t_{0},0).$
We want to show that 
\begin{equation}
\label{sursol}
F_{A}(p)\geq\lambda.
\end{equation}
By Lemma \ref{Tpentescri2}, if $\bar{p}$ is finite, then $\bar{p} \geq 0$ and
\begin{equation}
\label{exp1}
H(p+\bar{p})\geq\lambda. 
\end{equation}
If $\bar{p}=+\infty$ then as $H$ is coercive, the inequality (\ref{exp1}) is true replacing $\bar{p}$ with some large $\tilde{p}$. 
To simplify the notations,  $\bar{p}$ will denote the real number satisfying the inequality (\ref{exp1}) in the first or the second case.

As $F_{A}$ is non-increasing, we have
$$F_{A}(p)\geq F_{A}(p+\bar{p})$$ and
using Lemma \ref{positionHFA}, if $p+\bar{p}\notin \bigcup\limits_{\alpha\in I}]p_{\alpha},p_{\alpha}^{+}[$  we have 
$$F_{A}(p)\geq F_{A}(p+\bar{p})\geq H(p+\bar{p})\geq \lambda,$$ which prove the result. 
Now if $p+\bar{p}\in ]p_{\alpha},p_{\alpha}^{+}[$ for some $\alpha \in I$ such that $p_{\alpha}\neq p_{\alpha}^{+}$, then 
$$p+\bar{p} >p_{\alpha}=\phi_{\alpha}^{\prime}(0).$$
As for the sub-solution case, let us consider the modified test function
$$\varphi(t,x)=\phi(t,0)+\phi_{\alpha}(x)-\phi_{\alpha}(0).$$ Arguing as in the subsolution case, we can show that $\varphi$ touches $u$ from below at $(t_{0},0)$.

This test function satisfies in particular \eqref{changmfct} so we deduce that
$$ -\lambda+F_{A}(p_{\alpha})\geq 0,$$
so we have 
$$F_{A}(p+\bar{p})=F_{A}(p_{\alpha})\geq \lambda.$$
Therefore (\ref{sursol}) holds true.
\end{proof}

\subsection{Proof of the classification result}

To prove Theorem \ref{thclassification}, we first have to define the set limiter $A_{F}$ associated to the function  $F:\mathbb{R}\rightarrow\mathbb{R}$ continuous, non-increasing and semi-coercive \eqref{propF2}.

\begin{defin}[Set limiter $A_{F}$]
\label{defFAF}
The set limiter is $A_{F}$ the set of points $p\in \mathbb{R}$ such that either

\begin{equation}
\label{AFsous}
\left\{ \begin{array}{lll}
(i) & p^{-}\neq p,\\
(ii) & F(p) \geq H(p),\\
(iii) & \forall q\in \mathbb{R} \mbox{ such that }  F(q) \geq H(q) \mbox{ and }   ]q^{-},q^{+}[\cap]p^{-},p[\neq \emptyset,\\
 &  \mbox{ we have } H(q)\leq H(p),
\end{array}
\right.
\end{equation}
or 

\begin{equation}
\label{AFsur}
\left\{ \begin{array}{lll}
(i) & p^{+}\neq p,\\
(ii) & F(p) \leq H(p),\\
(iii) & \forall q\in \mathbb{R} \mbox{ such that }  F(q) \leq H(q) \mbox{ and }   ]q^{-},q^{+}[\cap]p,p^{+}[\neq \emptyset,\\
 &  \mbox{ we have } H(q)\geq H(p).
\end{array}
\right.
\end{equation}
\end{defin}

\begin{figure}
\begin{center}
  \includegraphics[width=15.5cm]{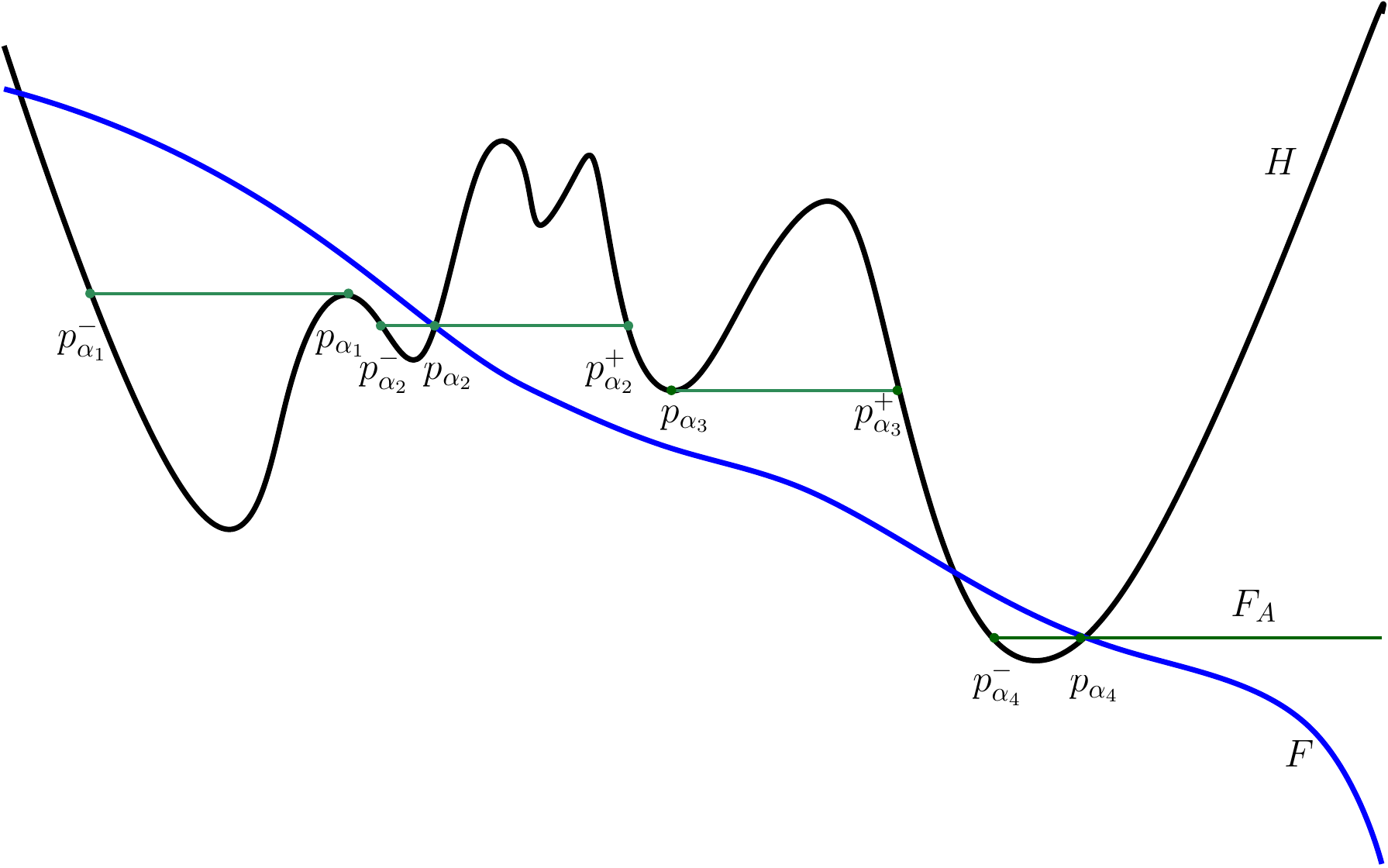}
  \caption{Illustration of a function $F_{A_{F}}$ in Definition \ref{defFAF}}
  \label{figFAF}
  \end{center}
  \end{figure}

We give an example of a $A_{F}$-limited flux function in Figure \ref{figFAF}.
To illustrate the set $A_{F}$, one can see that in the sets where $F\geq H$, the points of $A_{F}$ satisfying \eqref{AFsous} are local maximas. The sets where $F\leq H$, the points of $A_{F}$ satisfying \eqref{AFsur} are local minimas. The points of $A_{F}$ satisfying \eqref{AFsous} and \eqref{AFsur} are intersection points of $F$ with non-decreasing part of $H$ if $H$ has a finite number of minimas (see Figure \ref{figFAF}). 
We show that $p^{-}\neq p$ or $p^{+}\neq p$ for $p\in A_{F}$ characterizes the fact that $p$ satisfies \eqref{AFsous} or \eqref{AFsur}.

\begin{prop}
\label{AFfluxlim}
Let $F:\mathbb{R}\rightarrow\mathbb{R}$ be continuous, non-increasing and semi-coercive, and $A_{F}$ be defined as in Definition \ref{defFAF}, then $A_{F}$ is a set limiter. 
Moreover $A_{F}$ satisfies the following property.
If $p\in A_{F}$ and $p^{-}\neq p$ (resp. $p^{+}\neq p$) then $p$ satisfies \eqref{AFsous} (resp. \eqref{AFsur}).
In particular, if $p^{-}<p<p^{+}$, then $F(p)=H(p)$.
\end{prop}

\begin{proof}
Let us prove that $A_{F}$ is a set limiter.
The set $A_{F}$ satisfies 1. of Definition \ref{defA} since either $p^{-}\neq p$ or $p^{+}\neq p$.
Let us prove that it satisfies 2. and 3. of Definition \ref{defA}.

\vspace{0.5cm}
\textbf{Step 1: $A_{F}$ satisfies 2. of Definition \ref{defA}.}

Assume by contradiction that there exists $p_{1}, p_{2}\in A_{F}$ such that $p_{1}<p_{2}$ and $H(p_{1})<H(p_{2}).$
We distinguish four cases.

\vspace{0.25cm}
\textbf{Case 1: $p_{1}$ satisfies \eqref{AFsur}, $p_{2}$ satisfies \eqref{AFsous}}

We have $$ F(p_{1})\leq H(p_{1})<H(p_{2})\leq F(p_{2}).$$
But $F$ is non-increasing, so we get a contradiction and we have $H(p_{1})\geq H(p_{2})$.

\vspace{0.25cm}
\textbf{Case 2: $p_{1}, p_{2}$ satisfy \eqref{AFsous}}

Let $p=\inf\left\{q>p_{1} \mbox{  } | \mbox{  } H(q)\geq H(p_{2}) \right\}.$ We have $$p^{-}<p_{1}^{-}<p_{1}<p\leq p_{2}$$ and $$F(p)\geq F(p_{2})\geq H(p_{2})=H(p)>H(p_{1}).$$  So $p_{1}$ does not satisfy \eqref{AFsous} (iii) with $p$, that gives a contradiction.
	
\vspace{0.25cm}
\textbf{Case 3: $p_{1}, p_{2}$ satisfy \eqref{AFsur}}

Let $p=\sup\left\{q<p_{2} \mbox{  } | \mbox{  } H(q)\leq H(p_{1}) \right\}.$
By symmetry with case 2, we prove that $p_{2}$ does not satisfy \eqref{AFsur} (iii) and get a contradiction. 

\vspace{0.25cm}
\textbf{Case 4: $p_{1}$ satisfies \eqref{AFsous}, $p_{2}$ satisfies \eqref{AFsur}}

We have $F(p_{1})\geq H(p_{1})$ and $F(p_{2})\leq H(p_{2})$. 
Let us define $$q_{1}=\inf\left\{q\geq p_{1} \mbox{  } | \mbox{  } H(q)=F(q) \right\},$$
 $$r_{1}=\inf\left\{q\geq p_{1} \mbox{  } | \mbox{  } H(q)=H(q_{1}) \right\},$$
and $$q_{2}=\sup\left\{q\leq p_{2} \mbox{  } | \mbox{  } H(q)=F(q) \right\},$$
 $$r_{2}=\sup\left\{q\leq p_{2} \mbox{  } | \mbox{  } H(q)=H(q_{2}) \right\}.$$
Then if $H(r_{1})=H(q_{1})>H(p_{1})$, we have $$r_{1}^{-}<p_{1}^{-}<p_{1}<r_{1}$$ and $F(r_{1})\geq F(q_{1})=H(q_{1})=H(r_{1})$. So $p_{1}$ does not satisfy \eqref{AFsous} (iii) with $r_{1}$ that gives a contradiction. We deduce that $H(q_{1})\leq H(p_{1})$, so $$H(r_{2})=H(q_{2})=F(q_{2})\leq F(q_{1})=H(q_{1})\leq H(p_{1})<H(p_{2})$$ and we have $$r_{2}<p_{2}<p_{2}^{+}<r_{2}^{+},$$ and $F(r_{2})\leq F(q_{2})=H(q_{2})=H(r_{2})$.  So $p_{2}$ does not satisfy \eqref{AFsur} (iii) with $r_{2}$ that gives a contradiction. 

\vspace{0.5cm}
\textbf{Step 2: $A_{F}$ satisfies 3. of Definition \ref{defA}.}

Let $p\in \mathbb{R}$ such that $p^{-}\neq p^{+}$. We distinguish four cases.

\vspace{0.25cm}
\textbf{Case 1: $p^{-}\neq p$ and $F(p)<H(p)$.}

Let $p_{1}=\sup\left\{q\leq p \mbox{  } | \mbox{  } H(q)=F(q) \right\}$ and $p_{2}=\sup\left\{q \in [p_{1},p] \mbox{  } | \mbox{  } H(q)=\min\limits_{s\in [p_{1},p]} H(s) \right\}.$

The number $p_{1}$ could be $-\infty$ but as $H$ is coercive, $p_{2}<+\infty$. 

We are going to prove that $p_{2}\in A_{F}$ and $]p^{-},p^{+}[\cap]p_{2}^{-},p_{2}^{+}[ \neq \emptyset$.
Observe first that $p_{2}$ satisfies \eqref{AFsur} (i), (ii). Let us prove that it satisfies  \eqref{AFsur} (iii).
Assume by contradiction that there exists $q\in \mathbb{R}$ such that 
\begin{equation}
\label{ineq}
F(q)\leq H(q),
\end{equation} 

\begin{equation}
\label{empt}
]q^{-},q^{+}[\cap ]p_{2},p_{2}^{+}[\neq \emptyset
\end{equation} 
and 
\begin{equation}
\label{contrad}
H(q)<H(p_{2}).
\end{equation}
We distinguish three possibilities for $q$.
If $q<p_{1}$ then using \eqref{ineq} and \eqref{contrad}, we have $F(q)< H(p_{2})\leq H(p_{1})\leq F(p_{1})$, that gives a contradiction with the fact that $F$ is non-increasing.
If $q\in [p_{1},p]$ then by definition of $p_{2}$, $H(p_{2})\leq H(q)$ that gives a contradiction with \eqref{contrad}.
If $q>p$ then using \eqref{contrad}, we deduce that $q^{-}\geq p_{2}^{+}$ that gives a contradiction with \eqref{empt}.
We deduce that $p_{2}\in A_{F}$. Moreover, $p_{2}$ satisfies
\begin{equation}
\label{empt2}
]p^{-},p[\cap ]p_{2}^{-},p_{2}^{+}[\neq \emptyset.
\end{equation} 
Indeed, we have for $r\in]p^{-},p[, H(r)<H(p)$ by Lemma \ref{H-+}, so $H(p_{2})<H(p)$ and $p_{2}<p<p_{2}^{+}.$ 

\vspace{0.25cm}
\textbf{Case 2: $p^{-}\neq p$ and $F(p)\geq H(p)$.}

Let $p_{1}=\inf\left\{q\geq p \mbox{  } | \mbox{  } H(q)=F(q) \right\}$ and $p_{2}=\inf\left\{q \in [p,p_{1}] \mbox{  } | \mbox{  } H(q)=\max\limits_{s\in [p,p_{1}]} H(s) \right\}.$ We are going to prove that $p_{2}\in A_{F}$ and satisfies \eqref{empt2}.
We have $$p_{2}^{-}\leq p^{-}<p\leq p_{2}\leq p_{1},$$ so we deduce that $p_{2}$ satisfies \eqref{AFsous} (i) and by definition, we deduce that $p_{2}$ satisfies \eqref{AFsous} (ii). 
Let us prove that it satisfies  \eqref{AFsous} (iii).
Assume by contradiction that there exists $q\in \mathbb{R}$ such that 
\begin{equation}
\label{ineq2}
F(q)\geq H(q),
\end{equation} 
$q$ satisfies 
\begin{equation}
\label{empt4}
]q^{-},q^{+}[\cap ]p_{2}^{-},p_{2}[\neq \emptyset
\end{equation} 
and 
\begin{equation}
\label{contrad2}
H(q)>H(p_{2}).
\end{equation}
We distinguish three possibilities for $q$.
If $q>p_{1}$ then using \eqref{ineq2} and \eqref{contrad2}, we have $F(q)>F(p_{1})$, that gives a contradiction with the fact that $F$ is non-increasing.
If $q\in [p^{-},p_{1}]$ then $H(p_{2})\geq H(q)$ that gives a contradiction with \eqref{contrad2}.
If $q<p^{-}$ then $q^{+}\leq p_{2}^{-}$ that gives a contradiction with \eqref{empt4}.
We deduce that $p_{2}\in A_{F}$ and satisfies \eqref{empt2}. 

\vspace{0.25cm}
\textbf{Case 3: $p\neq p^{+}$ and $F(p)\leq H(p)$.}

Using the same arguments as in cases 1 and 2 with $p_{1}=\sup\left\{q\leq p \mbox{  } | \mbox{  } H(q)=F(q) \right\}$ and $p_{2}=\sup\left\{q \in [p_{1},p] \mbox{  } | \mbox{  } H(q)=\min\limits_{s\in [p_{1},p]} H(s) \right\}$, we deduce that $p_{2}\in A_{F}$ and satisfies 
\begin{equation}
\label{empt3}
]p,p^{+}[\cap ]p_{2}^{-},p_{2}^{+}[\neq \emptyset .
\end{equation}

\vspace{0.25cm}
\textbf{Case 4: $p\neq p^{+}$ and $F(p)> H(p)$.}

Using the same arguments as in cases 1 and 2 with $p_{1}=\inf\left\{q\geq p \mbox{  } | \mbox{  } H(q)=F(q) \right\}$ and $p_{2}=\inf\left\{q \in [p,p_{1}] \mbox{  } | \mbox{  } H(q)=\max\limits_{s\in [p,p_{1}]} H(s) \right\}$, we deduce that $p_{2}\in A_{F}$ and satisfies \eqref{empt3}. 

\vspace{0.5cm}
Now let us prove the property of $A_{F}$.
We only prove the result for $p^{+}\neq p$ since it is very similar for $p^{-}\neq p$. If $p$ satisfies \eqref{AFsur}, we are done. If $p$ satisfies \eqref{AFsous}, let us prove that it also satisfies \eqref{AFsur} in this case. By hypothesis, it satisfies \eqref{AFsur} (i). Let us prove that it satisfies \eqref{AFsur} (ii). Assume by contradiction that $F(p)>H(p)$. Consider $p_{2}$ defined in Step 2 Case 2. Then $p_{2}$ gives a contradiction with \eqref{AFsous} (iii), so $p$ satisfies \eqref{AFsur} (ii) and $F(p)=H(p)$. 

Now let us prove that $p$ satisfies \eqref{AFsur} (iii).
Assume by contradiction that there exists $q\in \mathbb{R}$ such that 

\begin{equation}
\label{internonvide}
]q^{-},q^{+}[\cup ]p,p^{+}[\neq \emptyset,
\end{equation}   

\begin{equation}
\label{inegal}
F(q)\leq H(q)
\end{equation}
and 

\begin{equation}
\label{inegalcont}
H(q)<H(p).
\end{equation}
We have that \eqref{inegal}, \eqref{inegalcont} implies $H(p)=F(p)>H(q)\geq F(q)$. So as $F$ is non-increasing, we have $q>p$ and Lemma \ref{disjoint} gives a contradiction with \eqref{internonvide}. We deduce the result.  
\end{proof}

The next lemma shows that the set $A_{F}$ associated to the function $F$ is uniquely determined. 

\begin{Lm}
\label{memeflux}
Let $A_{1}$ and $A_{2}$ be two set limiters. If $$\left\{ u \mbox{ }| \mbox{ } u \mbox{ solution of \eqref{eqHJ} with } F=F_{A_{1}} \right\}= \left\{ u \mbox{ }| \mbox{ } u \mbox{ solution of \eqref{eqHJ} with } F=F_{A_{2}} \right\},$$ then $$A_{1}=A_{2}.$$ 
\end{Lm}

\begin{proof}
Assume by contradiction that $A_{1}\neq A_{2}$. Then let $p_{\alpha_{1}}\in A_{1}$ such that $p_{\alpha_{1}}\notin A_{2}$. By 3. of Definition \ref{defA}, there exists $p_{\alpha_{2}}\in A_{2}$ such that 
\begin{equation}
\label{disj}
]p_{\alpha_{1}}^{-},p_{\alpha_{1}}^{+}[\cap ]p_{\alpha_{2}}^{-},p_{\alpha_{2}}^{+}[ \neq \emptyset.
\end{equation}
 So we have that $p_{\alpha_{1}}^{-}$ or $p_{\alpha_{1}}^{+}$ is in $]p_{\alpha_{2}}^{-},p_{\alpha_{2}}^{+}[$ or $p_{\alpha_{2}}^{-}$ or $p_{\alpha_{2}}^{+}$ is in $]p_{\alpha_{1}}^{-},p_{\alpha_{1}}^{+}[$ . We choose $p$ one of these elements.
The function $u(t,x)=-H(p)t+px$ is a solution of \eqref{eqHJ} for $F=F_{A_{1}}$ and for $F=F_{A_{2}}$ using the hypothesis. So we deduce that 
$$F_{A_{2}}(p)=H(p_{\alpha_{2}})=F_{A_{1}}(p)=H(p_{\alpha_{1}}).$$ Necessarily, as $p_{\alpha_{1}}\neq p_{\alpha_{2}}$, Lemma \ref{disjoint} gives a contradiction with \eqref{disj}. We deduce that $A_{1}=A_{2}$.
\end{proof}

%
%
%
%
%
%

Now we can deduce the main theorem \ref{thclassification} from the following proposition.

\begin{prop}[General Neumann boundary conditions reduce to flux-limited ones]
\label{propclassification}
Assume that the Hamiltonian $H:\mathbb{R}\rightarrow\mathbb{R}$ is continuous and coercive, the function $F:\mathbb{R}\rightarrow\mathbb{R}$ is continuous, non-increasing. Then there exists a set limiter $A_{F}$ such that 
\begin{itemize}
\item any relaxed super-solution of (\ref{eqHJ}) is an $A_{F}$-flux-limited super-solution;
\item any relaxed sub-solution of (\ref{eqHJ}) such that 
\begin{equation}
\label{ctefaible}
\forall t\in (0,T) u(t,0)=\limsup\limits_{(s,y)\rightarrow (t,0), y>0} u(s,y)
\end{equation}
is a $A_{F}$-flux-limited sub-solution;
\item any $A_{F}$-flux-limited sub-solution (resp. super-solution) is a $F$-relaxed sub-solution (resp. super-solution) of (\ref{eqHJ}). 
\end{itemize}
\end{prop}  

\begin{proof}[Proof of Theorem \ref{propclassification}]
We first prove that relaxed sub-solutions satisfying \eqref{continuitefaible} are flux-limited sub-solutions. We only do the proof for sub-solutions since it is very similar for super-solutions. Let $u$ be a relaxed sub-solution. 
Thanks to Theorem \ref{threduc}, it is enough to show that for all $\varphi$ touching $u^{*}$ from above at $(t,0)$ such that $\varphi_{x}(t,0)=p\in A_{F},$ and $p^{-}\neq p$, we have $$\varphi_{t}(t,0)+H(p)\leq 0.$$
Let $\varphi$ be such a test function. As $u$ is a relaxed sub-solution, we have $$\varphi_{t}+\min (F(p),H(p))\leq 0.$$ As $p^{-}\neq p $, Proposition \ref{AFfluxlim} implies $F(p)\geq H(p)$ so we deduce the result.

The second point of the theorem is a direct consequence of the inequality $$\min(F,H)\leq F_{A_{F}} \leq \max(F,H).$$
Indeed, if $p\in [p_{\alpha}^{-}, p_{\alpha}]$ where $p_{\alpha}\in A_{F}$ and $p_{\alpha}^{-}\neq p_{\alpha}$, using Proposition \ref{AFfluxlim}, and \eqref{AFsous} (ii), we have 
$$ F(p) \geq F(p_{\alpha})\geq H(p_{\alpha})=F_{A_{F}}(p)\geq H(p).$$
If $p\in [p_{\alpha}, p_{\alpha}^{+}]$ where $p_{\alpha}\in A_{F}$ and $p_{\alpha}^{+}\neq p_{\alpha}$, using Proposition \ref{AFfluxlim}, and \eqref{AFsur} (ii), we have 
$$ F(p) \leq F(p_{\alpha})\leq H(p_{\alpha})=F_{A_{F}}(p)\leq H(p).$$
If $p\notin \bigcup\limits_{\alpha \in I} [p_{\alpha}^{-},p_{\alpha}^{+}]$, then $H(p)=F_{A_{F}}(p)$.
\end{proof}

\begin{proof}[Proof of Theorem \ref{thclassification}]
Apply Proposition \ref{propclassification} and Lemma \ref{memeflux}.
\end{proof}

%

\begin{Lm}
\label{AFA=A}
Let $A$  be a set limiter.
The set limiter $A_{F_{A}}$ associated to the limited-flux function $F_{A}$ is the set $A$. In particular, a relaxed sub-solution (resp. super-solution) of \eqref{eqHJ} for $F=F_{A}$ is a flux-limited sub-solution (resp. super-solution) for $F=F_{A}$.  
\end{Lm}

\begin{proof}[Proof of Lemma \ref{AFA=A}]
Let us prove that $A\subset A_{F_{A}}$. Let $p\in A$. Without loss of generality, assume that $p^{-}\neq p$, so $p$ satisfies (i) of Definition \ref{defFAF}. By definition of $F_{A}$, we have $F_{A}(p)=H(p)$, so $p$ satisfies (ii) of Definition \ref{defFAF}.
Let us prove that $p$ satisfies (iii) of Definition \ref{defFAF}.
Assume by contradiction that there exists $q$ such that $F_{A}(q)\geq H(q)$ and 
\begin{equation}
\label{intersecpq}
]q^{-},q^{+}[\cap ]p^{-},p[\neq \emptyset,
\end{equation} 
and
\begin{equation}
\label{posipq}
H(p)<H(q).
\end{equation}
Then we deduce that $$F_{A}(p)=H(p)<H(q)\leq F_{A}(q),$$
so $q<p.$ We distinguish two cases, either $q\in ]p^{-},p[,$ or $q<p^{-}$. The first case is not possible since $q$ satisfies \eqref{posipq} which gives a contradiction with Lemma \ref{H-+}. So we have $q<p^{-}$. But \eqref{posipq} and Lemma \ref{H-+} imply that $q^{+}<p^{-}$, that gives a contradiction with \eqref{intersecpq}. So we have $A\subset A_{F_{A}}$. Using Proposition \ref{AFfluxlim}, $A_{F_{A}}$ is a set limiter.
Notice that if we add (resp. remove) an element to (resp. from) a set limiter, this new set is not a set limiter anymore. So necessarily, $A=A_{F_{A}}$ and we get the result. 
\end{proof}

\subsection{Comparison principle for a coercive Hamiltonian}

Using 1. of Theorem \ref{mainthCPnc} and Proposition \ref{propclassification}, we can deduce a comparison principle for a coercive Hamiltonian, but for $F$ only semi-coercive. 

\begin{proof}[Proof of 2. of Theorem \ref{mainthCPnc}]
We assume here that $F$ is semi-coercive \eqref{propF2}. 
We define $$p=\sup\left\{ q\in \mathbb{R} \mbox{ } | \mbox{ } H(q)=F(q) \right\},$$ and $G:\mathbb{R}\rightarrow \mathbb{R}$ a continuous function such that $G(x)\rightarrow -\infty$ when $x\rightarrow +\infty$, $G$ satisfies $G\leq F$ on $[p,+\infty[$. We define the function $\tilde{F}:\mathbb{R}\rightarrow\mathbb{R}$ such that 
$$ \tilde{F}=\left\{\begin{array}{ll}
F & \mbox{ on } ]-\infty,p] \\
G & \mbox{ on } [p,+\infty[.
\end{array} \right.$$
We have $A_{F}=A_{\tilde{F}}$. Indeed, notice that we have the following equivalences for $F$ and $\tilde{F}$, 
$$H(p)\leq F(p) \iff H(p)\leq \tilde{F}(p)$$
and $$H(p)\geq F(p) \iff H(p)\geq \tilde{F}(p).$$
Since in the definition of $A_{F}$, only the relative position between $F$ and $H$ takes the function $F$ into account, the previous equivalences give the result.
 So we deduce using Proposition \ref{propclassification} that a function $u$ is a relaxed sub-solution (resp. super-solution) for $F$ if and only if $u$ is a $A_{F}$-flux limited sub-solution (resp. super-solution), if and only if $u$ is a relaxed sub-solution (resp. super-solution) for $\tilde{F}$. We deduce the comparison principle for $F$ using the comparison principle for $\tilde{F}$ (1. of Theorem \ref{mainthCPnc}).
\end{proof}

\section{Comparison principle for nonconvex and noncoercive Hamilton-Jacobi equations allowing flat parts}

 In this section, we prove the first main comparison principle 1. of Theorem \ref{mainthCPnc} for a nonconvex and noncoercive Hamiltonian where the boundary condition allows flat parts. The proof follows the idea of coupling time and space in the doubling variable method in \cite{FIM}.
 First, we give a restricted version of the theorem which easily implies the main theorem. Then we prove the theorem for a class of test function which satisfy some properties. Finally, we give an example of such a test function so that the theorem is proven. 
 
\subsection{Simplification of the theorem}

Let us prove a restricted version of 1. of Theorem \ref{mainthCPnc} where the function $F$ satisfies more hypotheses. 

\begin{theo}[Restricted comparison principle] 
\label{mainthRes}
~\\
Assume that the Hamiltonian $H:\mathbb{R}\rightarrow\mathbb{R}$ is continuous, the function $F:\mathbb{R}\rightarrow\mathbb{R}$ is of class $\mathcal{C}^{1}$ and satisfies $F'<0$, $F(0)=0$ and \eqref{propF2}-\eqref{propF}, and the initial datum $u_{0}$ is uniformly continuous.
Then for all (relaxed) sub-solution $u$ and (relaxed) super-solution $v$ of (\ref{eqHJ})-(\ref{IC}) satisfying for some $T>0$ and $C_{T}>0,$
\begin{equation*}
u(t,x)\leq C_{T}(1+x), \quad v(t,x)\geq -C_{T}(1+x), \quad \forall (t,x)\in (0,T)\times [0,+\infty),
\end{equation*}
we have
$$u\leq v \quad \mbox{ in } \quad [0,T)\times [0,+\infty).$$
 
\end{theo}

\begin{proof}[Proof of 1. of Theorem \ref{mainthCPnc} using Theorem \ref{mainthRes}]
It is enough to assume $F(0)=0$  as in \cite[Lemma 3.1]{im}, by defining
$$u(t,x)=\tilde{u}(t,x)-tF(0) \quad \mbox{ and } \quad v(t,x)=\tilde{v}(t,x)-tF(0)$$ and $\tilde{F}=F-F(0)$, $\tilde{H}=H-F(0)$.
The function $u$ (resp. $v$) is a sub-solution (resp. super-solution) of \eqref{eqHJ} if and only if $\tilde{u}$ (resp. $\tilde{v}$) is a sub-solution (resp. super-solution) of \eqref{eqHJ} replacing $H$ by $\tilde{H}$ and $F$ by $\tilde{F}$.
Let the function $F$ be such that $F(0)=0$ and satisfy the hypothesis of 1. of Theorem \ref{mainthCPnc}, i.e. a continuous and non-increasing function which satisfies \eqref{propF2}-\eqref{propF}. By density, one can approximate $F$ by a sequence $F_{n}$ satisfying 
$$\left\| F_{n}-F \right\|_{\infty} \leq \frac{1}{n} \quad \forall n \in \mathbb{N}^{*},$$
with the hypothesis of Theorem \ref{mainthRes}, i.e. of class $\mathcal{C}^{1}$  and decreasing such that $F'<0$ which satisfies \eqref{propF2}-\eqref{propF}. Let $u$ be a sub-solution of (\ref{eqHJ}) with the function $F$.
Let us define $u_{n}=u(x)-\frac{t}{n}$ which is a sub-solution of (\ref{eqHJ}) with the function $F_{n}$ and $v_{n}=v(x)+\frac{t}{n}$ which is a super-solution of (\ref{eqHJ}) with the function $F_{n}$. Using Theorem \ref{mainthRes}, we deduce
$$u(t,x)-\frac{t}{n}\leq v(t,x)+\frac{t}{n} \quad \forall (t,x)\in [0,T)\times [0,+\infty).$$ Sending $n$ to $+\infty$, we deduce the result.
\end{proof}

\subsection{The coupling time and space test function}

\begin{theo}[Coupling time and space test function]
\label{functionphi}
Assume the function $F:\mathbb{R}\rightarrow\mathbb{R}$ is of class $\mathcal{C}^{1}$ and satisfies $F'<0$, $F(0)=0$ and \eqref{propF2}-\eqref{propF}. Then there exists a function $\varphi:\mathbb{R}^{2}\rightarrow \mathbb{R}$ of class $\mathcal{C}^{1}$ which satisfies the following properties.
\begin{enumerate}
\item (Superlinearity) 
\begin{equation}
\label{superlin}
\lim\limits_{\left| (t,x)\right|\rightarrow +\infty} \frac{\varphi(t,x)}{\left| (t,x)\right|}=+\infty,
\end{equation}
\item (Bounded from below) 
\begin{equation}
\label{bounded}
\forall (t,x)\neq (0,0), \quad \varphi(t,x)>\varphi(0,0)=0.
\end{equation}
\item (Differential inequalities) For all $t\in\mathbb{R},$
\begin{equation}
\label{diffineq}
\left\{\begin{array}{lll}
\varphi_{t}(t,x)+F(\varphi_{x}(t,x))&\geq 0 & \mbox{ if } x\leq 0, \\
\varphi_{t}(t,x)+F(\varphi_{x}(t,x))& \leq 0 & \mbox{ if } x\geq 0.
\end{array}\right.
\end{equation}
\end{enumerate}
\end{theo}

\begin{Rmk}
We first admit this theorem to prove the comparison principle and we show it in the next subsection. The idea of the proof is to replace in the doubling variable method, the usual term $\frac{(t-s)^{2}}{2\delta}+\frac{(x-y)^{2}}{2\epsilon}$ by $\delta\varphi\left(\frac{t-s}{\delta},\frac{x-y}{\delta}\right)$ which prevents the following supremum to be reached at the boundary.
\end{Rmk}

\subsection{Proof of the comparison principle}

Let us recall \cite[Lemma 3.4]{im} as we use it in the proof. The proof of this lemma is exactly the same as in \cite{im} so we skip it. 

\begin{Lm}[A priori control]
\label{aprioricontrol}
Let $T>0$ and let $u$ be a sub-solution and $v$ be a super-solution as in Theorem \ref{mainthRes}. Then there exists a constant $C=C(T)>0$ such that for all $(t,x),(s,y) \in [0,T)\times [0,+\infty)$, we have
\begin{equation*}
u(t,x)\leq v(s,y)+C(1+|x-y|).
\end{equation*}
\end{Lm}

\begin{proof}[Proof of Theorem \ref{mainthRes}]
The proof proceeds in several steps.
\vskip 0.5cm
\textbf{Step 1: Penalization procedure.}
We want to prove that 
$$M=\sup\limits_{(t,x)\in[0,T)\times [0,+\infty)}(u(t,x)-v(t,x))\leq 0.$$
Assume by contradiction that $M>0$. 
Let us define
$$M_{\delta,\alpha}=\sup\limits_{(t,x),(s,y)\in[0,T)\times [0,+\infty)}\left\{ u(t,x)-v(s,y)-\delta\varphi\left(\frac{t-s}{\delta},\frac{x-y}{\delta}\right)-\frac{\eta}{T-t}-\frac{\eta}{T-s}- \frac{\alpha x^{2}}{2} \right\}$$
where $\delta,\eta, \alpha$ are positive constants.
Then for $\alpha, \eta$ small enough, we have $M_{\delta,\alpha}\geq \frac{M}{2}>0.$
Indeed, by definition of the supremum $M$, there exists $(t_{0},x_{0})\in [0,T)\times [0,+\infty)$ such that 
$$u(t_{0},x_{0})-v(t_{0},x_{0})\geq \frac{3M}{4},$$ so 
$$M_{\delta,\alpha}\geq u(t_{0},x_{0})-v(t_{0},x_{0})-\frac{2\eta}{T-t_{0}}-\alpha \frac{x_{0}^{2}}{2}\geq \frac{M}{2},$$
 for $\alpha, \eta$ small enough. We want to show that this supremum is reached. 
For all $x,y,t,s$ such that
\begin{equation}
\label{useforIC}
0< \frac{M}{2}\leq u(t,x)-v(s,y)-\delta\varphi\left(\frac{t-s}{\delta},\frac{x-y}{\delta}\right)-\frac{\eta}{T-t}-\frac{\eta}{T-s}-\alpha \frac{x^{2}}{2},
\end{equation}
 by Lemma \ref{aprioricontrol}, we have  
 \begin{equation}
 \label{bornebis}
0< \frac{M}{2}\leq C_{T}(1+|x-y|)-\delta\varphi\left(\frac{t-s}{\delta},\frac{x-y}{\delta}\right)-\frac{\eta}{T-t}-\frac{\eta}{T-s}-\alpha \frac{x^{2}}{2},
\end{equation}
so we deduce that
\begin{equation}
\label{xyborne}
 \delta\varphi\left(\frac{t-s}{\delta},\frac{x-y}{\delta}\right)\leq C_{T}(1+|x-y|),
\end{equation}
and that 
\begin{equation}
\label{xborne}
\frac{(\alpha x)^{2}}{2}\leq \alpha C_{T}(1+|x-y|)
\end{equation}
By dividing (\ref{xyborne}) by $\left| (t-s,x-y) \right|$, the property (\ref{superlin}) of $\varphi$ implies that $x-y$ and $t-s$ are bounded, independently of $\alpha$, for $x, y, t, s$ satisfying (\ref{useforIC}). So using (\ref{xborne}), $x, y, t, s$ are in a compact set so the supremum $M_{\delta,\alpha}$ is reached at some point $(t,x,s,y)=(t_{\delta},x_{\delta},s_{\delta},y_{\delta})$.
Moreover, for $\delta\rightarrow 0$, using any converging subsequence and (\ref{xyborne}) dividing by $\left| (t-s,x-y) \right|$, using the property (\ref{superlin}) and (\ref{bounded}), we deduce that, $t_{\delta}-s_{\delta}$ and $x_{\delta}-y_{\delta}$ go to $0$. 

\vskip 0.5cm
\textbf{Step 2: Use of the initial condition.} We first treat the case where $t_{\delta}=0$ or $s_{\delta}=0$ along a subsequence.
If there exists a subsequence of $(t_{\delta},s_{\delta})$ converging to $(0,0)$ when $\delta\rightarrow 0,$ then calling $(x_{0},x_{0})$ any limit of subsequences of $(x_{\delta},y_{\delta})$, we get from (\ref{useforIC}),
$$0< \frac{M}{2}\leq u(t_{\delta},x_{\delta})-v(s_{\delta},y_{\delta}).$$
So letting $\delta\rightarrow 0$, the limit superior of the right hand side is smaller than $u_{0}(x_{0})-u_{0}(x_{0})=0$ and we get a contradiction. 

\vskip 0.5cm
\textbf{Step 3: Use of viscosity inequalities.} We can now assume that $t_{\delta}>0$ and $s_{\delta}>0$ and write the viscosity inequalities at $(t,x,s,y)=(t_{\delta},x_{\delta},s_{\delta},y_{\delta})$. 

\vskip 0.25cm
\textbf{Case 1:} If $x=0$ and $\min(H,F)=F$ at $\varphi_{x}\left(\frac{t-s}{\delta},\frac{-y}{\delta}\right).$

\vskip 0.25cm
The inequality for the sub-solution is
$$\frac{\eta}{(T-t)^2}+\varphi_{t}\left(\frac{t-s}{\delta},\frac{-y}{\delta}\right)+F\left(\varphi_{x}\left(\frac{t-s}{\delta},\frac{-y}{\delta}\right)\right)\leq 0.$$
Using property (\ref{diffineq}), we get a positive left-hand side which gives a contradiction.

\vskip 0.25cm
\textbf{Case 2:} If $y=0$ and $\max(H,F)=F$ at $\varphi_{x}\left(\frac{t-s}{\delta},\frac{x}{\delta}\right)$.

\vskip 0.25cm
The inequality for the super-solution is
$$-\frac{\eta}{(T-s)^2}+\varphi_{t}\left(\frac{t-s}{\delta},\frac{x}{\delta}\right)+F\left(\varphi_{x}\left(\frac{t-s}{\delta},\frac{x}{\delta}\right)\right)\geq 0.$$
Using property (\ref{diffineq}), we get a negatif left-hand side which gives a contradiction.

\vskip 0.25cm
\textbf{Case 3:} Other cases

\vskip 0.25cm
The inequality for the sub-solution is
$$\frac{\eta}{(T-t)^2}+\varphi_{t}\left(\frac{t-s}{\delta},\frac{x-y}{\delta}\right)+H\left(\varphi_{x}\left(\frac{t-s}{\delta},\frac{x-y}{\delta}\right)+\alpha x\right)\leq 0,$$
and the inequality for the super-solution is
$$-\frac{\eta}{(T-s)^2}+\varphi_{t}\left(\frac{t-s}{\delta},\frac{x-y}{\delta}\right)+H\left(\varphi_{x}\left(\frac{t-s}{\delta},\frac{x-y}{\delta}\right)\right)\geq 0.$$
Substracting these inequalities, we get
\begin{equation}
\label{compatible}
\frac{2\eta}{T^{2}} \leq H\left(\varphi_{x}\left(\frac{t-s}{\delta},\frac{x-y}{\delta}\right)\right)-H\left(\varphi_{x}\left(\frac{t-s}{\delta},\frac{x-y}{\delta}\right)+\alpha x\right).
\end{equation}
As $t-s$ and $x-y$ are bounded independently of $\alpha$ and as $\alpha x$ goes to $0$ when $\alpha\rightarrow 0,$ thanks to \eqref{xborne}, using the fact that $H$ is uniformly continuous in a compact, the right hand side of (\ref{compatible}) goes to $0$ when $\alpha\rightarrow 0,$ we get a contradiction. The proof is now complete. 

\end{proof}

\subsection{Construction of the test function}

The idea is to construct a test function coupling time and space, of the form
$$\varphi(t,x)=f(t)+g(x)+xE(t),$$
where the functions $f, g, E:\mathbb{R}\rightarrow\mathbb{R}$ are of class $\mathcal{C}^{1}$. In this section, the function $F$ satisfies the hypothesis of Theorem \ref{mainthRes}. Let us first define a function $G$, we will next use it to define the function $E$.

\begin{defin}[Function $G$]
Let $G$ be a continuous function such that
\begin{itemize}
\item $G \geq \max((-F^{-1})',(-2F)^{-1})>0,$
\item $G$ is even i.e. $\forall t\in \mathbb{R}, \quad G(-t)=G(t)$,
\item $G$ is non-increasing in $(-\infty,0]$ and non-decreasing on $[0,+\infty).$
\end{itemize}
\end{defin}

\begin{Rmk}
The function $G$ exists as $\max((-F^{-1})',(-2F)^{-1})$ is continuous and $(-F^{-1})'$ is positive. Moreover, we have
$$\lim\limits_{x\rightarrow \pm \infty} G(x)= +\infty,$$
since $(-2F)^{-1}$ is increasing and goes to $+\infty$ at $+\infty$.
\end{Rmk}

\begin{prop}[Function $E$]
\label{functionE}
Assume $F$ is of class $\mathcal{C}^{1}$ and satisfies $F'<0$, $F(0)=0$ and \eqref{propF2}-\eqref{propF}. Then there exists a function $E$ of class $\mathcal{C}^{1}$ solution of the ODE
\begin{equation}
\label{ODE}
\left\{ \begin{array}{l}
E'=\frac{1}{G(-2F(E))} \\
E(0)=0,
\end{array}\right.
\end{equation}
 which satisfies the same properties as $-F$, i.e., $E'>0$,  $E(0)=0$ and 
 \begin{equation}
 \label{propE}
\lim_{x\rightarrow -\infty}E(x)=-\infty \quad \mbox{ and } \quad \lim_{x\rightarrow +\infty}E(x)=+\infty.
\end{equation}
Moreover, we have
 \begin{equation}
\label{limitE'}
\lim_{x\rightarrow \pm \infty}E'(x)=0.
\end{equation}
\end{prop}

\begin{proof}[Proof of Proposition \ref{functionE}] The existence of a solution for \eqref{ODE} is given by Cauchy-Peano-Arzela global existence theorem. 
Indeed, as $0<(-F^{-1})'(0)\leq G$, we have $0<\frac{1}{G}\leq\frac{1}{(-F^{-1})'(0)}$ so the function
$$ \frac{1}{G(-2F)}$$ is bounded and continuous.  
Moreover, as $G \geq (-F^{-1})'>0$, we have $E'>0$.
Let us prove that $E$ satisfies (\ref{propE}) by contradiction. If $E$ has a finite limit then using (\ref{ODE}), $E'$ has a finite limit $L>0$ so 
$$E(t)\sim Lt  $$ and $E$ has an infinite limit which is a contradiction. We deduce (\ref{limitE'}) using (\ref{ODE}).
\end{proof}

Let us define the function $f$.
\begin{defin}[Function $f$]
Let $f$ be the function of class $\mathcal{C}^{1}$ such that $f'(t)=-F(E(t))$ and $f(0)=0$.
\end{defin}

Let us define the function $g$. 
 First, we define some functions $\psi$, $\psi_{1}$ and $\psi_{2},$ 
 \begin{align*}
 \psi(t,x) &=-F^{-1}(xE'(t)-F(E(t))-E(t),\\
\psi_{1}(x) &=\sup_{t\in\mathbb{R}} \psi(t,x),\\
 \psi_{2}(x) &=\inf_{t\in\mathbb{R}} \psi(t,x).
  \end{align*}
 
 \begin{prop}
 \label{psi12}
 The function $\psi_{1}$ is lower semi-continuous and locally bounded in $[0,+\infty)$, continuous at $0$ and satisfies $\psi_{1}(0)=0$.   
  The function $\psi_{2}$ is upper semi-continuous and locally bounded in $(-\infty,0]$, continuous at $0$ and satisfies $\psi_{2}(0)=0$. 
 \end{prop}
 
 \begin{proof}[Proof of Proposition \ref{psi12}]
The function $\psi_{1}$ (resp. $\psi_{2}$) is lower (resp. upper) semi-conti\-nuous because it is a supremum (resp. infimum) of continuous functions.

Let us prove that $\psi_{1}$ and $\psi_{2}$ are locally bounded and continuous at $0$.
By using the Taylor expansion of the function $-F^{-1}$ of class $\mathcal{C}^{1}$, there exists $\theta:\mathbb{R}^{2}\rightarrow[0,1]$ such that 
$$\psi(t,x)=xE'(t)(-F^{-1})'(-F(E(t))+\theta(t,x)xE'(t)).$$

If $0\leq x \leq R$, for $R>0$, as $G\geq (-F^{-1})'>0$, we have
\begin{equation}
\label{ingcont}
 \begin{array}{lll}
0\leq \psi(t,x)& \leq & xE'(t)G(-F(E(t))+\theta(t,x)xE'(t))\\
 &\leq & xE'(t)G(-F(E(t))+RE'(t)).
\end{array}
\end{equation}
Let us prove that the continuous function $h:t\rightarrow E'(t)G(-F(E(t))+RE'(t))$ is bounded in $\mathbb{R}$. Since $h$ is continuous, we only need to prove that $h$ is bounded for $|t|$ big enough.  
Using (\ref{limitE'}), for $t\geq 0$ big enough, we have $RE'(t)\leq 1$ and $-F(E(t))+1\leq -2F(E(t)).$ 
 Using that $G$ is non-decreasing in $[0,+\infty)$, we deduce from (\ref{ODE}) that  
 $$0\leq h(t) \leq E'(t)G(-F(E(t))+1)\leq \frac{G(-F(E(t))+1)}{G(-2F(E(t)))}\leq 1.$$
 By the same argument,  for $t\leq 0$ small enough, we have $RE'(t)\geq -1$ and $-F(E(t))-1\geq -2F(E(t))$. So as $G$ is non-increasing in $(-\infty,0]$, we deduce with (\ref{ODE}) that  
 $$0\leq h(t) \leq E'(t)G(-F(E(t))-1)\leq 1.$$ 
 We deduce from \eqref{ingcont} that $\psi_{1}$ is locally bounded in $[0,+\infty)$ and that $\psi_{1}(0)=0$. By the same arguments, we also deduce that $\psi_{2}$ is locally bounded in $(-\infty,0]$ and that $\psi_{2}(0)=0$.
The proof is now complete.
\end{proof}
 
 \begin{Lm}[Function $g$]
 Let $g$ be a function of class $\mathcal{C}^{1}$ such that $g(0)=0$ and such that $g'$ satisfies $g'(0)=0$ and
 $$g'(x)\geq \max(2x,\psi_{1}(x)) \quad \mbox{ for } x\geq 0,$$ and $$g'(x)\leq \min(2x,\psi_{2}(x)) \quad \mbox{ for } x\leq 0.$$
 \end{Lm}
 
 \begin{proof}
 The construction of the function $g'$ is a consequence of the fact that $\psi_{1}$ and $\psi_{2}$ are locally bounded and continuous at $0$. 
 \end{proof}
 
 Now, we can prove that the function $\varphi$ defined by $\varphi(t,x)=f(t)+g(x)+xE(t)$ satisfies (\ref{diffineq}).
 
 \begin{prop}
 \label{ineqbis}
 The function $\varphi(t,x)=f(t)+g(x)+xE(t)$ satisfies (\ref{diffineq}).
 \end{prop}
 
 \begin{proof}[Proof of Proposition \ref{ineqbis}]
 As the function $g$ satisfies for all $t\in\mathbb{R}$,
  $$g'(x)\geq \psi_{1}(x)\geq \psi(t,x)=(-F^{-1})(xE'(t)-F(E(t))-E(t) \quad \mbox{ for } x\geq 0,$$ and  $$g'(x)\leq \psi_{2}(x)\leq \psi(t,x)=(-F^{-1})(xE'(t)-F(E(t))-E(t) \quad \mbox{ for } x\leq 0,$$
  and as $-F^{-1}$ is increasing, we deduce that 
  $$-F(E(t))+xE'(t)+F(g'(x)+E(t))\leq 0 \quad \mbox{ for } x\geq 0,$$ and $$-F(E(t))+xE'(t)+F(g'(x)+E(t))\geq 0 \quad \mbox{ for } x\leq 0.$$
  These inequalities are exactly  (\ref{diffineq}).
 \end{proof}

Let us prove that the function $\varphi$ satisfies (\ref{superlin}) and (\ref{bounded}).
\begin{prop}
\label{propphisuperlin}
The function $\varphi$ is of class $\mathcal{C}^{1}$ and superlinear \eqref{superlin}.
\end{prop}

\begin{proof}[Proof of Proposition \ref{propphisuperlin}]
By construction, the function $\varphi$ is of class $\mathcal{C}^{1}$. With the definition of $g$ in hand, we deduce that $g(x)\geq x^{2}$. Using that 
$$|xE(t)|\leq \frac{x^{2}}{2}+\frac{E(t)^{2}}{2},$$  we deduce that 
\begin{equation}
\label{inegsuper}
 \begin{array}{lll}
\varphi(t,x) & \geq & f(t)+x^{2}-\frac{E(t)^{2}}{2}-\frac{x^{2}}{2},\\
 & \geq & f(t) -\frac{E(t)^{2}}{2}+\frac{x^{2}}{2}.
\end{array}
\end{equation}
Let us prove that $\frac{E^{2}}{2f}$ goes to $0$ when $|t|\rightarrow +\infty$. We first compare their derivative which are simpler.
We have
\begin{equation}
\label{quotderiv}
\begin{array}{lll}
\frac{2f'(t)}{(E^{2})'(t)}=\frac{-F(E(t))}{E'(t)E(t)} &= &
\frac{-F(E(t))G(-2F(E(t)))}{E(t)},\\
& \geq & \frac{-F(E(t))(-2F)^{-1}(-2F(E(t)))}{E(t)}\\
&\geq & -F(E(t)).
\end{array}
\end{equation}
where the last term goes to $+\infty$ as $t$ goes to $+\infty$.
We have the same result for $t\leq 0$ using the same argument and the fact that $G$ is even,
$$\frac{2f'(t)}{(E^{2})'(t)}\geq F(E(t)),$$
where the last term goes to $+\infty$ as $t$ goes to $-\infty$.
We deduce that
$$\frac{(E^{2})'(t)}{f'(t)}\rightarrow 0 \quad \mbox{ for } t\rightarrow \pm \infty.$$
As $\int_{0}^{t} E^{2'}(s) \, \mathrm{d}s=E^{2}(t)$ diverges when  $t\rightarrow \pm \infty$,
we have $$\frac{\int_{0}^{t} (E^{2})'(s) \, \mathrm{d}s}{\int_{0}^{t} f'(s) \, \mathrm{d}s} \rightarrow 0,$$ so
$$\frac{E(t)^{2}}{f(t)} \rightarrow 0 \quad \mbox{ for } t\rightarrow \pm \infty.$$ 
And as $f$ is superlinear \eqref{superlin}, $t\rightarrow f(t)-\frac{E(t)^{2}}{2}$ is superlinear. 
We deduce, from \eqref{inegsuper} that $\varphi$ satisfies (\ref{superlin}).
\end{proof}

\begin{prop}
\label{propphibounded}
The function $\varphi$ satisfies (\ref{bounded}).
\end{prop}

\begin{proof}[Proof of Proposition \ref{propphibounded}]
The function $\varphi$ is of class $\mathcal{C}^{1}$, satisfies $\varphi(0,0)=0$ and is superlinear \eqref{superlin} in $(t,x)$. 
Let us prove that its local extremum is reached only at the point $(0,0)$ and this implies (\ref{bounded}).
Let $(t,x)\in \mathbb{R}^{2}$ satisfy, 
\begin{equation}
\label{ptcrit}
\left\{\begin{array}{lllll}
\varphi_{t}(t,x) & = &-F(E(t))+xE'(t) & = & 0\\
\varphi_{x}(t,x) & = & g'(x)+E(t) & = & 0. 
\end{array}\right.
\end{equation}
First, we notice that for $(t,x)$ satisfying \eqref{ptcrit}, $t=0$ if and only if $x=0$.
Let us prove that $t=0$ as soon as $x>0$ and $(t,x)$ satisfies \eqref{ptcrit}.
 If $x> 0$, we have taking $s=0$
 $$\begin{array}{lll}
 -E(t)= g'(x) &\geq & \sup\limits_{s\in \mathbb{R}}\left\{(-F^{-1}(xE'(s)-F(E(s)))-E(s)\right\}\\
  & \geq & -F^{-1}(xE'(0)),
\end{array}$$
so we have 
$$E(t)\leq F^{-1}(xE'(0)).$$
And we also have, as $F$ is decreasing,
$$xE'(t)=F(E(t))\geq F(F^{-1}(xE'(0)))=xE'(0).$$
If $t\geq 0$, as $E'$ is non-increasing in $[0,+\infty)$, we deduce that $t\leq0$ so $t=0$ and $x=0$, which gives a contradiction. 
If $t\leq 0$, as $E'$ is non-decreasing, we deduce that $t\geq0$ so $t=0$ and $x=0$, which also gives a contradiction. 
The case $x<0$ is similar so we skip it. This ends the proof.
\end{proof}

\begin{proof}[Proof of Theorem \ref{functionphi}]
Combine Propositions \ref{ineqbis}, \ref{propphisuperlin} and \ref{propphibounded}.
\end{proof}

\appendix

\section{Reformulation of state constraints}
Let us prove the reformulation of state constraint result in the case where the Hamiltonian is not necessarily convex.  

\begin{theo}[Reformulation of state constraints] 
\label{reforSC}
Assume $H:\mathbb{R}\rightarrow\mathbb{R}$ is continuous and coercive \eqref{coerc} and $u:(0,T)\times [0,+\infty) \rightarrow\mathbb{R}$
satisfies \eqref{ctefaible}
then $u$ is a viscosity solution of 
\begin{equation}
\label{sens1}
\left\{ \begin{array}{lll}
u_{t}+H(u_{x})= 0 & \mbox { in } & (0,T)\times  (0,+\infty)\\
u_{t}+H(u_{x})\geq 0 & \mbox { in } & (0,T)\times \{0 \},
\end{array}\right.
\end{equation} 
if and only if $u$ is a viscosity solution of the flux-limited problem
\begin{equation}
\label{sens2}
\left\{ \begin{array}{lll}
u_{t}+H(u_{x}) = 0 & \mbox { in } & (0,T)\times (0,+\infty)\\
u_{t}+H^{-}(u_{x})= 0 & \mbox { on } & (0,T)\times \{0 \},
\end{array}\right.
\end{equation} 
where $H^{-}$ is the decreasing part of the Hamiltonian defined by
$$H^{-}(p)=\inf\limits_{q\leq p} H(q).$$ 
\end{theo}

First we prove that $F_{A_{H^{-}}}=H^{-}$ that allows us to use Theorem \ref{threduc} of reduction of the set of test functions.

\begin{defin}[Set limiter $A_{0}$]
\label{setA0}
Let $H:\mathbb{R}\rightarrow\mathbb{R}$ be continuous and coercive \eqref{coerc}. The set limiter $A_{0}$ is the set of points $p\in \mathbb{R}$ such that 
\begin{itemize}
\item $p^{-}=p<p^{+},$
\item $\forall q\in \mathbb{R} \mbox{ such that } ]q^{-},q^{+}[\cap]p,p^{+}[\neq \emptyset, \mbox{ we have } H(q)\geq H(p)$.
\end{itemize}
\end{defin}

\begin{Lm}
\label{setA0H-}
 We have $A_{H^{-}}=A_{0}$. 
\end{Lm}

\begin{proof}[Proof of Lemma \ref{setA0H-}]
Notice first that $H^{-}\leq H$ and that $H^{-}$ is non-increasing.
Using Definition \ref{defFAF}, it only remains to prove that for all $p\in A_{H^{-}}$ we have $p^{-}=p$.
Assume by contradiction that there exists $p\in A_{H^{-}}$ such that $p^{-}<p$. Then using Proposition \ref{AFfluxlim} we deduce that $p$ satisfies (ii) of \eqref{AFsous} so $H(p)=H^{-}(p)$. We deduce from Lemma \ref{H-+} that 
$$\forall q\in ]p^{-},p[ \quad H^{-}(q)\leq H(q)< H(p)=H^{-}(p),$$ but $H^{-}$ is non-increasing which gives a contradiction. So we have $p^{-}=p$. 
We deduce that $A_{H^{-}}=A_{0}$.  
\end{proof}

\begin{Lm}
\label{fluxA0H-}
We have $F_{A_{H^{-}}}=F_{A_{0}}=H^{-}$.
\end{Lm}

\begin{proof}[Proof of Lemma \ref{fluxA0H-}]
From Lemma \ref{setA0H-}, we deduce that $F_{H^{-}}=F_{A_{0}}$. Let us prove that $F_{A_{0}}=H^{-}$. 
Notice first that 
\begin{equation}
\label{inegfluxA0}
F_{A_{0}}\leq H. 
\end{equation}
Let $p\in\mathbb{R}$. 

If there exists $p_{\alpha}\in A_{0}$ such that $p\in ]p_{\alpha},p_{\alpha}^{+}[$ then we have 
$$H^{-}(p)\leq F_{A_{0}}(p)=H(p_{\alpha}).$$
Moreover, from Lemma \ref{H-+} we have 
$$\forall q\in ]p_{\alpha},p[ \quad H(p_{\alpha})<H(q)$$ and as $F_{A_{0}}$ is non-increasing and by \eqref{inegfluxA0}, we have also 
$$\forall q\leq p_{\alpha} \quad H(p_{\alpha})=F_{A_{0}}(p_{\alpha})\leq F_{A_{0}}(q)\leq H(q).$$
So we have 
$$H^{-}(p)=\inf\limits_{q\leq p} H(q) =H(p_{\alpha})=F_{A_{0}}(p).$$

If $p\notin \bigcup\limits_{p_{\alpha}\in A_{0}} ]p_{\alpha},p_{\alpha}^{+}[$, then $$F_{A_{0}}(p)=H(p)\geq H^{-}(p).$$
Moreover, as $F_{A_{0}}$ is non-increasing and by \eqref{inegfluxA0}, we have 
$$ \forall q\leq p \quad H(p)=F_{A_{0}}(p)\leq F_{A_{0}}(q)\leq H(q).$$
 So $F_{A_{0}}(p)=H(p)=H^{-}(p)$.  We deduce that $F_{A_{0}}=H^{-}$. 
\end{proof}

The proof is exactly the same as in \cite{gue1,im}.  

\begin{proof}[Proof of Theorem \ref{reforSC}]
We do the proof in three steps.
\newline
\textbf{1st step:}
Let us prove that
$$u_{t}+H(u_{x}) \leq 0 \quad \mbox{ in } (0,T)\times (0,+\infty),$$
implies
$$u_{t}+H^{-}(u_{x}) \leq 0\quad \mbox{ on } (0,T)\times\{0\}.$$
Since $\forall p_{\alpha}\in A_{0}$, $p_{\alpha}^{-}=p_{\alpha}$, using Theorem \ref{threduc}, we deduce that $u$ is a $A_{0}$-flux limited sub-solution, so
$$u_{t}+F_{A_{0}}(u_{x}) \leq 0 \quad \mbox{ on } (0,T)\times\{0\}.$$ 
As $F_{A_{0}}(u_{x})=H^{-}(u_{x})$, we have $$u_{t}+H^{-}(u_{x})\leq 0\quad \mbox{ on } (0,T)\times\{0\}.$$
\textbf{2nd step:}
Let us prove that
$$u_{t}+H(u_{x}) \geq 0\quad \mbox{ in } (0,T)\times [0,+\infty),$$
implies
$$u_{t}+H^{-}(u_{x}) \geq 0 \quad\mbox{ on } (0,T)\times\{ 0\}.$$
Let $\varphi$ be a test function touching $u_{*}$ from below at $(t_{0},0).$
Using Theorem \ref{threduc}, we assume that
$$\varphi(t,x)=\psi(t)+\phi_{\alpha}(x),$$
where $\psi\in \mathcal{C}^{1}((0,T))$
and $$\phi_{\alpha}\in\mathcal{C}^{1}([0,+\infty)), \quad \phi'_{\alpha}(0)=p_{\alpha}.$$
We have $\varphi_{x}(t_{0},0)=p_{\alpha}$ and
 $$H(\varphi_{x}(t_{0},0))=H(p_{\alpha})=F_{A_{0}}(p_{\alpha})=H^{-}(p_{\alpha})=H^{-}(\varphi_{x}(t_{0},0)),$$
so by hypothesis, we have $\varphi_{t}+H(\varphi_{x}(t_{0},0))\geq 0$.
We deduce that $$\varphi_{t}+H^{-}(\varphi_{x}(t_{0},0))\geq 0.$$
\textbf{3rd step:}
The reverse come from the fact that $H^{-}\leq H.$
\end{proof}

\begin{Rmk}
In \cite{gue1}, the author gives simpler proofs without using Theorem of reduction of the set of test functions which can be adpated for a nonconvex Hamiltonian in dimension 1 for the stationary case. 
\end{Rmk}

%
%

\textbf{Acknowledgements.} The author thanks R. Monneau for giving the idea of coupling time and space in the test function for the doubling variable method. The author thanks also C. Imbert for all his advice and support concerning this work. This work was partially supported by the ANR-12-BS01-0008-01 HJnet project.

\bibliography{bibli}
\bibliographystyle{plain}

\end{document}